\documentclass[10pt]{article}

\pdftrailerid{}
\usepackage[utf8]{inputenc}
\usepackage[T1]{fontenc}
\usepackage{amsmath,amssymb,dsfont,mathrsfs}
\numberwithin{equation}{section}
\usepackage{microtype}
\usepackage{graphicx,tikz,pgfplots}
\graphicspath{{images/}}
\pgfplotsset{compat=newest}
\usepackage[hyperref,amsmath,thmmarks]{ntheorem}
\usepackage{aliascnt}
\usepackage[a4paper,centering,bindingoffset=0cm,marginpar=2cm,margin=2.5cm]{geometry}
\usepackage[pagestyles]{titlesec}
\usepackage[font=footnotesize,format=plain,labelfont=sc,textfont=sl,width=0.75\textwidth,labelsep=period]{caption}
\usepackage{comment}
\usepackage{multirow}
\usepackage{subcaption}
\usepackage{algorithm, algpseudocode}
\usepackage{booktabs}
\usepackage{array}
\usepackage{makecell}

\makeatletter

\usepackage[backend=biber,maxnames=10,backref=true,hyperref=true,giveninits=true,safeinputenc]{biblatex}
\DefineBibliographyStrings{english}{%
	backrefpage = {cited on page},
	backrefpages = {cited on pages},
}

\DeclareFieldFormat[report]{title}{``#1''}
\DeclareFieldFormat[book]{title}{``#1''}
\AtEveryBibitem{\clearfield{url}}
\AtEveryBibitem{\clearfield{note}}

\titleformat{\section}[block]{\large\sc\filcenter}{\thesection.}{0.5ex}{}[]
\titleformat{\subsection}[runin]{\bf}{\thesubsection.}{0.5ex}{}[.]

\usepackage[pdftex,colorlinks=true,linkcolor=blue,citecolor=green,urlcolor=blue,bookmarks=true,bookmarksnumbered=true]{hyperref}
\hypersetup
{
    bookmarksnumbered,
    breaklinks,
    pdfborderstyle=,
    colorlinks,
    citecolor=[rgb]{0,.65,0},
    linkcolor=[rgb]{0,0,.5},
    urlcolor=[rgb]{0,0,.5},
    pdfauthor={Authors},%ADAPT
    pdfsubject={Subject},%ADAPT
    pdftitle={Title},%ADAPT
    pdfkeywords={Keywords}%ADAPT
}

\newpagestyle{headers}
{
	\headrule
	\sethead[\footnotesize\thepage][\footnotesize\sc Authors][]{}{}{\footnotesize\thepage}
	\setfoot{}{}{}
}
\newtheorem{lemma}{Lemma}[section]

\newaliascnt{proposition}{lemma}

\aliascntresetthe{proposition}

\newaliascnt{corollary}{lemma}
\newtheorem{corollary}[corollary]{Corollary}
\aliascntresetthe{corollary}

\newaliascnt{theorem}{lemma}
\newtheorem{theorem}[theorem]{Theorem}
\aliascntresetthe{theorem}

\theorembodyfont{\normalfont}
\newaliascnt{definition}{lemma}
\newtheorem{definition}[definition]{Definition}
\aliascntresetthe{definition}

\newaliascnt{assumption}{lemma}

\aliascntresetthe{assumption}

\theorembodyfont{\normalfont}
\newaliascnt{example}{lemma}
\newtheorem{example}[example]{Example}
\aliascntresetthe{example}

\newaliascnt{notation}{lemma}
\newtheorem{notation}[notation]{Notation}
\aliascntresetthe{notation}

\newaliascnt{convention}{lemma}

\aliascntresetthe{convention}

\newaliascnt{remark}{lemma}
\newtheorem{remark}[remark]{Remark}
\aliascntresetthe{remark}

\theoremstyle{nonumberplain}
\theoremseparator{:}
\theoremheaderfont{\normalfont\itshape}

\theoremsymbol{\ensuremath{\square}}
\newtheorem{proof}{Proof}

\newcommand{\N}{\mathds{N}}

\newcommand{\R}{\mathds{R}}

\newcommand{\dom}[1]{\mathcal{D}\left(#1\right)} 

\let\RE\Re
\let\Re=\undefined
\DeclareMathOperator{\Re}{\RE e}
\let\IM\Im
\let\Im=\undefined
\DeclareMathOperator{\Im}{\IM m}
\newcommand{\abs}[1]{\left|#1\right|}
\newcommand{\norm}[1]{\left\|#1\right\|}
\newcommand{\norms}[1]{\|#1\|}
\newcommand{\set}[1]{\left\{#1\right\}}
\newcommand{\inner}[2]{\left<#1,#2\right>}

\let\ii\i
\renewcommand{\i}{\mathrm i}
\renewcommand{\d}{\,\mathrm d}

\newcommand{\ve}{\varepsilon}

\newcommand{\vw}{\vec{w}}

\newcommand{\bs}{\vec{s}}
\newcommand{\bt}{\vec{t}}

\newcommand{\bw}{{\vec{w}}}

\newcommand{\X}{{\tt X}}
\newcommand{\Y}{{\tt Y}}

\newcommand{\V}{{\tt V}}

\newcommand{\y}{{\tt y}}
\newcommand{\yd}{\y^\delta}

\newcommand{\m}{{\tt m}}
\newcommand{\n}{{\tt n}}
\newcommand{\x}{{\tt x}}

\newcommand{\xad}{\x^{\alpha\delta\eta}}
\newcommand{\xdad}{\xad_{\m\n}}
\newcommand{\xdag}{\x^\dagger}
\newcommand{\z}{{\tt z}}

\newcommand{\op}[1]{{F[#1]}}

\newcommand{\ulx}{\underline{\x}}

\newcommand{\uly}{\underline{\y}}
\newcommand{\no}{{\tt N}}

\newcommand{\Tik}{\mathbf{T}}
\newcommand{\Reg}{\mathbf{R}}

\renewcommand{\arraystretch}{1.5}

\usepackage{rotating}

\title{Neural operators for solving nonlinear inverse problems}
\author{Otmar Scherzer$^{1,2,3}$\\{\footnotesize\href{mailto:otmar.scherzer@univie.ac.at}{otmar.scherzer@univie.ac.at}}
\and Thi Lan Nhi Vu$^{1,3}$\\
{\footnotesize\href{mailto:thi.lan.nhi.vu@univie.ac.at}{thi.lan.nhi.vu@univie.ac.at}}
\and Jikai Yan$^{1}$ \\{\footnotesize\href{mailto:jikai.yan@univie.ac.at}{jikai.yan@univie.ac.at}}
}
\date{}

\begin{document}

%'
%' titlepage
%'
\maketitle
\thispagestyle{empty}
% ADAPT
\begin{center}
\parbox[t]{17em}{\footnotesize
\hspace*{-1ex}$^1$Faculty of Mathematics\\
University of Vienna\\
Oskar-Morgenstern-Platz 1\\
A-1090 Vienna, Austria}
\hfil
\parbox[t]{17em}{\footnotesize
\hspace*{-1ex}$^2$Johann Radon Institute for Computational\\
\hspace*{1em}and Applied Mathematics (RICAM)\\
Altenbergerstraße 69\\
A-4040 Linz, Austria}
\end{center}

% if publication is from CD-LAB MaMSi
\begin{center}
\parbox[t]{17em}{\footnotesize
\hspace*{-1ex}$^3$Christian Doppler Laboratory\\
\hspace*{1em}for Mathematical Modeling and Simulation\\
\hspace*{1em}of Next Generations of Ultrasound Devices\\
\hspace*{1em}(MaMSi)\\
Oskar-Morgenstern-Platz 1\\
A-1090 Vienna, Austria}
\end{center}

\setcounter{section}{0}

\begin{abstract}		
	We consider solving a probably infinite dimensional operator equation, where the operator is not modeled by physical laws but is specified indirectly via training pairs of the input-output relation of the operator. 
	Neural operators have proven to be efficient to approximate infinite dimensional operators. 
	In this paper we analyze Tikhonov regularization with neural operators as surrogates for solving ill-posed operator equations. The analysis is based on balancing approximation errors of neural operators, regularization parameters, and noise. Moreover, we extend the approximation properties of neural operators from sets of continuous functions to Sobolev and Lebesgue spaces, which is crucial for solving inverse problems and we discuss the problem of finding an appropriate network structure of neural operators (training). Finally, we present some numerical experiments.  
\end{abstract}

\section{Introduction}
In this paper, we consider \emph{nonlinear inverse problems}, which are formulated as finding the solution of a \emph{nonlinear} operator equation
\begin{equation} \label{eq:op}
	F[\x]=\y\;,
\end{equation}
where the operator $F: \mathcal{D}(F) \subseteq \X \to \Y$ is defined on its domain $\mathcal{D}(F) \neq \emptyset$. For technical reasons, we need to 
	consider the operator $F$ in parallel in different function space settings. With the notation, we mean that we consider mapping properties of $F$ (such as continuity) with respect to the topologies of $\X$ and $\Y$.

Inverse problems are often ill-posed, meaning that their solutions do not depend continuously on the data $\y$. This means that even if $\y^\delta \to \y$, there is no guarantee that a solution $\x^\delta$ of \autoref{eq:op} with data $\y^\delta$ (presumably it exists) will converge to a solution of \autoref{eq:op}. In this paper, the superscript $\delta$ means that we have some estimate on the amount of noise,
\begin{equation}\label{eq:magnitude}
	\norms{\y-\yd}_\Y \leq \delta\,.
\end{equation}
Therefore, a regularization technique needs to be applied to solve \autoref{eq:op} in a stable manner. One of the most widely studied regularization method is Tikhonov regularization, which approximates a solution of \autoref{eq:op} by a minimizer $\x^{\alpha\delta}$ of the quadratic Tikhonov functional 
\begin{equation} \label{eq:tik}
\Tik^{\alpha\delta}[\x]:=\norms{F[\x]-\yd}_\Y^2 + \alpha \norms{\x-\x^{(0)}}_\X^2 \quad \text{over} \quad  \x \in \mathcal{D}(F), 
\end{equation}
with $\alpha > 0$ for some fixed point $\x^{(0)} \in \mathcal{D}(F)$. It is common to put $\x^{(0)} = 0$ when dealing with linear operator, and we do the same here. Tikhonov regularization has been generalized to convex regularization, which consists in minimization of 
\begin{equation} \label{eq:tikc}
	\Tik_\mathcal{R}^{\alpha\delta}[\x]:=\norms{F[\x]-\yd}_\Y^2 + \alpha \Reg[\x] \quad \text{over} \quad  \x \in \mathcal{D}(F), 
\end{equation}
where $\Reg: \X \to [0,\infty]$ is a functional, typically convex and proper. The most important generalized variational regularization methods are \emph{total variation minimization} (see \cite{RudOshFat92}) and \emph{sparsity regularization} (see \cite{DauDefMol04}). A general convex regularization functional is often associated with a formulation in a Banach space (see for instance \cite{SchGraGroHalLen09,SchuKalHofKaz12}). We do not consider these generalizations, because the considered simple examples have been evaluated in more depth in the quadratic setting and for Hilbert spaces. Generalizations can of course be done but lead to technical complications (see \cite{PoeResSch10}).  

For computational purposes, the Tikhonov functional is discretized in the following way:
\begin{enumerate}
	\item The space $\X$ is approximated by a sequence of nested subspaces $\set{\X_\m : \m \in \N}$ such that 
	\begin{equation*}
		\overline{\bigcup_{\m \in \N} \X_\m} = \overline{\lim_{\m \to \infty} \X_\m} = \X\;.
	\end{equation*} 
    \item A family of operators $\set{F_\n: \n \in \N}$ is constructed to approximate $F$ uniformly in a bounded neighborhood of a solution $\xdag$ of \autoref{eq:op}. Specifically, there exists a ball $\mathcal{B}_{\tt r}(\xdag)$ with radius ${ \tt r} > 0$ such that
    \begin{equation} \label{eq:unifrom}
       \norms{F[\x] - F_\n[\x]}_\Y \leq \rho_\n \quad \text{for all} \quad \x \in \mathcal{D}(F) \cap \mathcal{B}_{\tt r}(\xdag),
    \end{equation}
    where $\rho_\n \to 0$ as $\n \to \infty$. 
\end{enumerate}

For fixed $\eta > 0$, discretized Tikhonov regularization consists in computing 
\begin{equation} \label{eq:xdad} 
	\xdad \in \X_\m \cap \mathcal{D}(F)
\end{equation}
which satisfies 
\begin{equation} \label{eq:eta}
\norms{F_\n[\xdad]-\yd}_\Y^2 + \alpha\norms{\xdad-\x^{(0)}}_\X^2 \leq \inf_{\x \in \X_\m \cap \mathcal{D}(F)} \set{
   \norms{F_\n[\x]-\yd}_\Y^2 + \alpha\norms{\x-\x^{(0)}}_\X^2} + \eta\;.
\end{equation} 
We call $\xdad$ approximate minimizer with tolerance $\eta$ of the functional (see \cite{NeuSch90,PoeResSch10})
\begin{equation} \label{eq:discTik}
	\Tik_{\m\n}^{\alpha\delta}[\x] :=
		\norms{F_\n[\x]-\yd}_\Y^2 + \alpha\norms{\x-\x^{(0)}}_\X^2\;.
\end{equation}
In this paper, we focus on investigating the use of surrogate operator approximations, such as neural operators learned from training data, in Tikhonov regularization, and, for simplicity of presentation, we leave out the approximations of the elements of $\X$. Therefore, we consider the analysis of approximate minimizers $\x_\n^{\alpha\delta\eta} \in \mathcal{D}(F) \subseteq X$ of the functional  
\begin{equation} \label{eq:discTik_s}
	\Tik_{\n}^{\alpha\delta}[\x] := \norms{F_\n[\x]-\yd}_\Y^2 + \alpha\norms{\x-\x^{(0)}}_\X^2\;.
\end{equation}

Traditionally, spline spaces have been used for $\X_\m$, and finite element operators have been used for $F_\n$ (see \cite{Neu89,NeuSch90,PoeResSch10}). In this setting, discretizations $\m$ and $\n$ can be determined such that $\x_{\m\n}^{\alpha\delta\eta}$ is an optimal approximation in $\X_\m$ of the minimum-norm solution $\xdag$ of \autoref{eq:op}. The conceptual difference of neural operators is that they are learned from training data (see \autoref{eq:S} below).

Neural operators, as used in this paper, are defined as DeepONet's in \cite{LuJinPanZhaKar21}. The idea of neural operator traces back to \cite{CheChe93,TiaHon95}, where the concepts of learning functionals and operators were first developed. Building on the DeepONet framework and inspired by the physics informed neural networks \cite{RaiPerKar19}, other type of neural operators have been developed in \cite{WanWanPer21}. In parallel with DeepONet-based implementations, the Fourier neural operator has been introduced in \cite{LiKovAziLiuBha20_report}.  

The \emph{quantitative} error estimates of neural operators are not as advanced as those for finite difference and finite element approximation operators. Therefore, in turn, \emph{quantitative} regularization theory, as established for classical approximation methods in \cite{Neu89,NeuSch90,PoeResSch10}, needs to be developed, which is a concern of this paper. To motivate the challenges, we study simple examples from the inverse problems literature (see \cite{EngKunNeu89,EngHanNeu96}). The outline of the paper is as follows:
\begin{enumerate}
	\item \textbf{Examples of inverse problems:} First, we recall two examples of inverse problems, originally proposed in \cite{EngKunNeu89}, which demonstrate the optimal interactions between discretizations $\m$, $\n$, noise $\delta$, and the regularization parameter $\alpha$. We consider these examples first in the classic setting with finite element approximations $F_\n$ (see \autoref{se:examples}), and subsequently with neural operators learned from training samples (see \autoref{sec:basisrep} and \autoref{se:generalization}).
	\item \textbf{Challenges of neural operators:} We recall the definition of \emph{neural operators} from \cite{LuJinPanZhaKar21} and highlight some challenges when applying them as part of a regularization method (see \autoref{se:deeponet}). We propose modifications to neural operators that enable their use as surrogate operators for approximately solving \autoref{eq:op} (see \autoref{sec:basisrep} and \autoref{se:generalization}). For this paper, due to the Hilbert space setting, we require network operators to be defined on Sobolev and Lebesgue spaces. 
	\item \textbf{Priors for neural operators:} Training a neural operator involves solving for a high-dimensional set of parameters. Since the training process is highly nonlinear, strategies for determining appropriate priors are essential to identify locally optimal training parameters (see for instance \cite{SchHofNas23}). One such strategy is discussed in \autoref{se:orthogonal}.  
	\item In \autoref{sec:numerics}, we present some numerical results for the basic examples introduced in \autoref{se:examples}.
\end{enumerate}

\section{General assumptions} \label{sec:assump}
Throughout this paper we use the following notation:
\begin{notation} \label{not:xy}
	$\X = \set{\x: \Omega_\X \subseteq \R^m \to \R}$ and $\Y=\set{\y:\Omega_\Y \subseteq \R^n \to \R}$ are spaces of functions with norms $\norms{\cdot}_\X$ and $\norms{\cdot}_\Y$, respectively. We assume that $\Omega_\X$ and $\Omega_\Y$ are bounded domains with piecewise smooth boundaries and that they satisfy the cone property (see \cite{Ada75}). We refer to $\Y$ as the \emph{data space} and $\X$ as the \emph{image space} in accordance with the terminology of \cite{AspKorSch20,AspFriKorSch21}. 
\end{notation}

Moreover, we make the following assumptions:
\begin{itemize} 
	\item The operator $F$ is \emph{weakly sequentially closed} between Hilbert spaces $\X$ and $\Y$, meaning that, if the sequence $(\x_k)_{k \in \N} \subseteq \mathcal{D}(F)$ converges weakly to $\x$ in $\X$ and $(F[\x_k])_{k \in \N} \subseteq \Y$ converges weakly to $\y$ in $\Y$, then $\x \in \mathcal{D}(F)$ and $F[\x]=\y$.
	\item The operator $F$ is continuous, meaning that for every sequence $(\x_k)_{k \in \N} \subseteq \dom{F}$, if $\x_k \to \x$ in $\X$ and $\x \in \dom{F}$, then $F[\x_k] \to F[\x]$ in $\Y$.
\end{itemize}
	Under the assumptions of weak sequential closedness and continuity of $F$, it follows that:
\begin{enumerate}
	\item If \autoref{eq:op} has a solution, then there exists an $\x^{(0)}$-minimum-norm solution $\xdag$. That is, 
	\begin{equation} \label{eq:argmin}
		\xdag := \text{argmin} \set{ \norms{\x-\x^{(0)}}_\X : \x \in \mathcal{D}(F), \x \text{ solves } \autoref{eq:op}}\;.
	\end{equation}
	\item For every $\alpha,\delta >0$, the Tikhonov functional $\Tik^{\alpha\delta}$ defined in \autoref{eq:tik} attains a minimizer (is well-posed), it is stable and convergent in the sense of a regularization method (see \cite[Theorem 10.2 \& Theorem 10.3]{EngHanNeu96}).
\end{enumerate}
Moreover, we assume that 
\begin{itemize}
	\item the operator $F$ is Fr\'echet differentiable in an open ball $\mathcal{B}_{\tt r}(\xdag) \subseteq \mathcal{D}(F) \subseteq \X$ with ${\tt r} > 2\norm{\xdag - \x^{(0)}}_\X$, and that
	\item the Fr\'echet derivative is Lipschitz continuous in $\mathcal{B}_{\tt r}(\xdag)$. That is, there exist $L_1 > 0$ such that for all $\x \in \mathcal{B}_{\tt r}(\xdag)$ 
	\begin{equation} \label{eq:lipschitz}
		\norms{F'[\x]-F'[\xdag]} \leq L_1 \norms{\x-\xdag}_\X  \text{ for all } \norms{\x - \xdag}_\X \leq {\tt r}\;,
	\end{equation}
	Here $\norms{F'[\x]-F'[\xdag]}$ denotes the operator norm of $F'[\x]-F'[\xdag]$.
	\end{itemize}

The Fr\'echet differentiability is used to prove convergence rates of Tikhonov regularized solutions (see \cite[Theorem 10.4]{EngHanNeu96}). This is the main concern of this paper.
We use the following lemma in the course of the text: 
\begin{lemma} \label{le:lipschitz}
	Let $F$ be Fr\'echet differentiable with Lipschitz continuous Fr\'echet derivative in $\mathcal{B}_{\tt r}(\xdag)$.
	Denote by $L^\dagger = \norm{F'[\xdag]}$ the operator norm of $F'[\xdag]$, where $\xdag$ is an $\x^{(0)}$-minimum-norm solution, then for all $\x,\x_1,\x_2 \in \mathcal{B}_{\tt r}(\xdag)$ we have
	\begin{equation} \label{eq:norm}
		\norm{F'[\x]} \leq L_0
	\end{equation}
	and
	\begin{equation} \label{eq:lipschitz0}
		\norm{F[\x_1]-F[\x_2]}_\Y \leq L_0 \norm{\x_1-\x_2}_\X \text{ with } L_0=L^\dagger + L_1{\tt r}\;.
	\end{equation}
\end{lemma}
\begin{proof}
For all $\z \in \X$ with $\norm{\z}_\X \leq 1$ we have
\begin{equation*}
	\begin{aligned}
		\norm{F'[\x] \z}_\Y &\leq \norm{F'[\xdag] \z}_\Y + \norm{(F'[\x]-F'[\xdag]) \z}_\Y
		\leq \left( L^\dagger + L_1 \norm{\x-\xdag}_\X \right) \norm{\z}_\X\,,
	\end{aligned}
\end{equation*}	
which gives the assertion of \autoref{eq:norm}.
Moreover, we get from \autoref{eq:norm} that
\begin{equation} \label{eq:l0} \begin{aligned}
		\norm{F[\x_1]-F[\x_2]}_\Y &\leq \int_0^1 \norm{F'[\x_1+t(\x_2-\x_1)](\x_2-\x_1)}_\Y \mathrm{d}t \\&
		\leq \int_0^1 \norm{\left(F'[\x_1+t(\x_2-\x_1)-F'[\xdag]\right) (\x_2-\x_1)}_\Y \mathrm{d}t + \norm{F'[\xdag](\x_2-\x_1)}_\Y\\
		&\leq L_1 \left(\int_0^1 \norm{\x_1+t(\x_2-\x_1)-\xdag}_\X \mathrm{d}t\right) \norm{\x_2-\x_1}_\X + L^\dagger\norm{\x_2-\x_1}_\X \;.
	\end{aligned}
\end{equation}
Now, we note that 
\begin{equation*}
	\begin{aligned}
		\int_0^1 \norm{\x_1+t(\x_2-\x_1)-\xdag}_\X \mathrm{d}t 
		&= \int_0^1 \norm{(1-t)(\x_1-\xdag) +t (\x_2-\xdag)}_\X \mathrm{d}t\\
		&\leq \norm{\x_1-\xdag}_\X \int_0^1 (1-t) \mathrm{d}t  + \norm{\x_2-\xdag}_\X
		\int_0^1t \mathrm{d}t \leq {\tt r}\,.
	\end{aligned}
\end{equation*}
Therefore, we get from \autoref{eq:l0} that
\begin{equation*}\begin{aligned}
	\norm{F[\x_1]-F[\x_2]}_\Y &\leq (L^\dagger + L_1{\tt r}) \norm{\x_2-\x_1}_\X\;.
\end{aligned}
\end{equation*}
\end{proof}
Before we go deeper in the analysis of neural operators, we bring some motivating examples.

\section{Examples (classics with finite elements)} \label{se:examples}

In this section, we review two simple inverse problems and how they are solved with Tikhonov regularization with classical finite element operator approximations. 

\begin{example}[The $a$-example]\label{ex: darcyflow}
	For a given ${\tt f} \in  L^2(0,1)$, we consider the differential equation for $\y$:
	\begin{equation} \label{eq:darcyBVP}
		\begin{cases}
			\quad -\left(\x \y'\right)'(s) &= {\tt f}(s) \text{ for all }s \in (0,1),\\
			\quad  \y(0) = \y(1) &= 0,
		\end{cases}		
	\end{equation}
	where $\x \in \X:=H^1(0,1)$ is from the image space. The parameter $\x$ is often denoted by $a$ in the literature (see \cite{EngKunNeu89,NeuSch90}), which motivates our denomination of the example. 
	For this problem, the forward operator is defined as
	\begin{equation}\label{eq:darcyflow}
		\begin{aligned}
			F: \mathcal{D}(F):= \set{\x \in H^1(0,1): \x \geq \nu>0} & \to L^2(0,1)\;\\
			\x & \mapsto \y =: F[\x].
		\end{aligned}
	\end{equation}
    We first consider the classical \emph{finite element approach}: 
	Let $\Y_\n$ denote the space of linear splines on a uniform mesh with mesh size $\n^{-1}$, vanishing at the boundary. The finite element approximation of \autoref{eq:darcyBVP} is $\y_\n \in \Y_\n$, which satisfies:
	\begin{equation}\label{eq:disc}
		\inner{\x\y_\n'}{\tt y'}_{L^2} = \inner{\tt f}{\tt y}_{L^2} \quad \text{for all} \quad \tt y \in \Y_\n.
	\end{equation}
    
	The finite element approximation $F_\n$ of the operator $F$ is given by:
	\begin{equation}\label{eq:fea}
		\begin{aligned}
			F^{\text{FE}}_\n: \mathcal{D}(F):= \set{\x \in H^1(0,1): \x \geq \nu>0} & \to L^2(0,1)\;\\
			\x & \mapsto \y_\n=: F^{\text{FE}}_\n[\x].
		\end{aligned}
	\end{equation}
	
	The approximation error of $F^{\text{FE}}_\n$ has been intensively studied and is given by (see, for instance, \cite{Aub67,Nit68,Cia78}):
	\begin{equation}\label{eq:a-aproxi-error}
		\norms{F[\x] - F^{\text{FE}}_\n[\x]}_{L^2} = \norms{\y-\y_\n}_{L^2} = \mathcal{O}((1+\norms{\x}_{H^1})\n^{-2})\;.
	\end{equation}
	That is, for ${\tt r} > 0$ fixed, we have  
	\begin{equation*}
		\rho_\n :=  \sup_{\x \in \mathcal{B}_{\tt r}(\xdag)} \norms{F[\x] - F^{\text{FE}}_\n[\x]}_{L^2} = \mathcal{O}(\n^{-2}).
	\end{equation*}
	With such estimates, it is possible to derive optimal convergence rates for $\x_{\m\n}^{\alpha\delta}$, which we recall here for the readers convenience. We summarize the essential assumptions for \cite[Corollary 3.5]{NeuSch90}:

    \begin{itemize}
        \item Let $\xdag\in H^2(0,1)$ and assume that $\tt f$ is sufficiently smooth such that $\y=\y^0=F[\xdag] \in H^2(0,1)\cap H^1_0(0,1)$ (this is the noise-free data). 
        \item We further assume that the \emph{source condition} (as postulated in \cite{EngKunNeu89,NeuSch90}) holds:
	   \begin{equation}\label{de:source_a}	
		  \begin{array}{c}
			\varphi:=\xdag-\x^{(0)} \in H^3(0,1) \text{ satisfies } \varphi'(0)=\varphi'(1)=0,\\
       r \in (0,1) \mapsto \int_0^r \frac{(I-\Delta)(\x^{(0)}-\xdag)}{\y'[\xdag]}(s) ds \in H^2(0,1)\cap H_0^1(0,1) \text{ has sufficiently small }L^2\text{ norm.}
            \end{array}
            \end{equation}
\end{itemize}
With $\eta$ as in \autoref{eq:eta}, we make the choice of parameters
		\begin{equation} \label{eq:par}
		 \alpha \sim \max \set{\delta,\n^{-2}}, \; \eta \sim \alpha^2,
		\end{equation}
	    which gives the optimal convergence rate for the regularized solution (as defined in \autoref{eq:xdad}): 
	    \begin{equation} \label{eq:es}
	    	\norms{\x_\n^{\alpha\delta\eta} - \xdag}_{H^1} = \mathcal{O}(\sqrt{\delta}+\n^{-1})\;.
	    \end{equation}
	\end{example}

\begin{example}[The $c$-example] \label{ex:2point} Let ${\tt f} \in L^2(0,1)$ be given. Now, we consider the differential equation
	\begin{equation}
		\begin{cases}\label{eq:invillpos}
			\quad -\y''(s) + \x(s)\y(s) &= {\tt f}(s) \text{ for all }s \in (0,1),\\
			\quad  \quad \quad \quad \y(0) = \y(1) &= 0,
		\end{cases}		
	\end{equation}
	with $\x \in L^2(0,1)$ from the set of images. $\x$ is often denoted by $c$ in the literature, which gives this example its name. Here, the forward operator is defined as 
	\begin{equation*}
		\begin{aligned}
			F: \mathcal{D}(F):= \set{\x \in L^2(0,1): \x \geq 0 \text{ a.e.} } & \to L^2(0,1)\;.\\
			\x & \mapsto \y =: F[\x]
		\end{aligned}
	\end{equation*}
	The finite element approximation of the operator $F$ on the finite-dimensional subspace of linear spline $\Y_\n$ with $\n-1$ internal nodes of $\Y = L^2(0,1)$ is given by
	\begin{equation} \label{eq:fec}
		\begin{aligned}
			F^{\text{FE}}_\n: \mathcal{D}(F):= \set{\x \in L^2(0,1): \x \geq 0 \text{ a.e.}} & \to L^2(0,1)\;.\\
			\x & \mapsto \y_\n=: F^{\text{FE}}_\n[\x]
		\end{aligned}
	\end{equation}
	Here, $\y_\n\in\Y_\n$ is the unique solution of 
	\begin{equation*}
		\inner{\y_\n'}{\tt y'}_{L^2}+\inner{\x\y_\n}{\tt y}_{L^2} = \inner{f}{\tt y}_{L^2} \quad \text{ for all } \quad \tt y\in \Y_\n.
	\end{equation*}

In the general convergence rates analysis from \cite{NeuSch90}, an additional assumption that $\xdag$ is in the interior of $\dom{F}$ is used. For every $\xdag \in \dom{F} \subseteq L^2(0,1)$, this assumption cannot be verified due to the constraint that $\xdag \geq 0$. However, it was also shown there that, for this particular example, the assumption of being an interior point can be circumvented by the condition 
$0 < \xdag \in H^2(0,1)$. Note that every $H^1(0,1)$ function is uniformly continuous in $[0,1]$ and therefore, it is strictly positive. Moreover, we assume that $\xdag$ satisfies 
    \begin{equation}\label{de:source_c}
     \frac{\x^{(0)}-\xdag}{\y[\xdag]} \in H^2(0,1) \cap H_0^1(0,1) \text{ with } \text{a sufficiently small $L^2$ norm}.
    \end{equation}
    Analogously to \autoref{ex: darcyflow}, we see that for ${\tt r} > 0$ fixed, we have  
    \begin{equation*}
    	\rho_\n :=  \sup_{\x \in \mathcal{B}_{\tt r}(\xdag)} \norms{F[\x] - F^{\text{FE}}_\n[\x]}_{L^2} = \mathcal{O}(\n^{-2}).
    \end{equation*}
    This suggests to choose 
	\begin{equation} \label{eq:par_c}
		\alpha \sim \max \set{\delta,\n^{-2}}, \; \eta \sim \alpha^2,
	\end{equation}
	such that we obtain the following convergence rates for the regularized solutions: 
	\begin{equation} \label{eq:es_c}
		\norms{\x_\n^{\alpha\delta\eta} - \xdag}_{L^2} = \mathcal{O}(\sqrt{\delta}+\n^{-1})\;.
	\end{equation}
\end{example}

We emphasize that in \cite{NeuSch90}, additional results with various kinds of smoothness assumptions have been shown.

The goal of this paper is to derive error estimates of the form \autoref{eq:es} based on neural operators. 

\begin{remark}
    The source conditions \autoref{de:source_a} and \autoref{de:source_c} differ by the operator $(I-\Delta)$, which appears in \autoref{de:source_a} but not in \autoref{de:source_c}. The difference arises from the domain setting of the operator $F$. The source conditions are based on the existence of $\omega \in \Y$ such that $\xdag - \x^{(0)} = F'[\xdag]^*\omega$. In the $a$-example, $\x \in H^1(0,1)$, and $(I-\Delta)^{-1}$ arises when computing the adjoint $F'[\xdag]^*: L^2(0,1) \to H^1(0,1)$ since $(I-\Delta)^{-1}$ is the adjoint of the embedding from $H^1$ into $L^2$. In the $c$-example, $\x \in L^2(0,1)$, so no non-trivial embedding operator is needed. For more details see \cite{EngKunNeu89}.
\end{remark}

\section{Neural functions, functionals and operators} \label{se:deeponet}

In the following, we review universal approximation theorems with neural functions, functionals, and operators. The latter are studied in \cite{TiaHon95}, where they are considered to be defined on spaces of continuous functions. The approximations are defined based on neural networks, which in turn are defined via activation functions (see, for instance, \cite{Cyb89,KirMejPerSchShi25_preprint}). Commonly used activation functions are of sigmoid type: 
\begin{definition}[Sigmoid function]\label{def:sigm_act} A strictly monotonically increasing and differentiable function $\sigma: \R \to \R$ is called \emph{sigmoidal} if it satisfies
	\begin{equation}\label{eq:sigm_act} \sigma(t) \to 
		\begin{cases}
			1 & \text{as } t \to +\infty \\
			0 & \text{as } t \to -\infty 
		\end{cases}.
	\end{equation}
\end{definition}
For examples, the 
\begin{itemize}
	\item logistic sigmoid function $t \in \R \to \text{sig}(t) =(1+\exp(-t))^{-1}$, 
	\item the scaled and shifted hyperbolic tangens $t \in \R \to \tanh(t)/2+1/2$ and  
	\item the scaled and shifted arctangens $t \in \R \to \text{atan}(t)/\pi+1/2$
\end{itemize}
are sigmoidal. 
Sigmoid functions have been introduced as a universal tool to approximate nonlinear functions (see \cite{Cyb89,HorStiWhi89b}) as well as functionals and operators (see \cite{CheChe93,TiaHon95}).

\subsection{Approximation of functions}
In the following, we review an approximation result with neural functions in the classical setting of continuous functions. These results are based on compact sets:
\begin{definition}[Compactness] \label{def:compact} 
	Let $(\V,\norms{\cdot}_\V)$ be a normed space and let ${\tt K} \subseteq \V$ be a subset. We say that ${\tt K}$ is compact in $\V$ if every sequence in ${\tt K}$ has a subsequence that converges in the $\norms{\cdot}_\V$ norm to an element in ${\tt K}$.
\end{definition}
In the following, we state some approximation results for continuous functions with neural network functions.

\begin{theorem}[Approximation of functions (Theorem 3, \cite{TiaHon95})] \label{th:function apx} 
	Let $\sigma$ be a  sigmoid activation function as defined in \autoref{def:sigm_act}. Furthermore let $\Omega$ be a bounded domain in $\R^n$ and let $\V=C(\overline{\Omega})$ be the Banach space of continuous functions equipped with the supremum norm $\norms{\cdot}_{L^\infty(\Omega)}$. Moreover, let ${\tt K}$ be a compact subset of $\V$.
	Then, for every $\ve >0$, there exist a number $J := J(\ve) \in \N$ and coefficients 
	$$\zeta_j := \zeta_j(\ve) \in \R, \; \bw_j:=\bw_j(\ve) \in \R^n, \; j = 1,\dots,J\,,$$ 
	such that for every ${\tt u} \in {\tt K}$ there exists coefficients 
	$$c_j := c_j(\ve,{\tt u}) \in \R, \; j = 1,\dots,J\,,$$ 
	such that the function
	\begin{equation*}
		\bt \in \overline{\Omega} \mapsto 
		{\tt u}_\n(\bt):=\sum_{j=1}^{J} c_j\sigma\left(\bw_j^T\bt + \zeta_j\right)
	\end{equation*}
	satisfies  
	\begin{equation}\label{eq:uniform}
	\norms{{\tt u}-{\tt u}_\n}_{L^\infty(\Omega)} \leq \ve\;.
	\end{equation}
	We summarize the coefficient of ${\tt u}_\n$ in a vector:
	\begin{equation*}
		\n =\abs{\mathcal{T_\n}} = J(2+n)
		\text{ and }
		\mathcal{T}_\n = \begin{pmatrix} \underbrace{ c_{j}}_{\in \R} & \underbrace{ \bw_j}_{\in \R^n} & \underbrace{ \zeta_j}_{\in \R} &\end{pmatrix}_{\tiny \begin{array}{c} j=1,\dots,J\end{array}}.
	\end{equation*}
\end{theorem}
\begin{remark}
 We emphasize that the estimate \autoref{eq:uniform} holds uniformly for all elements ${\tt u}$ on the compact set ${\tt K}$ of $\V$.
\end{remark}

In the following, we review approximation of functionals and operators defined on spaces of continuous functions with \emph{neural functionals and operators:}
\subsection{Approximation of functionals}

The universal approximation theorem for functions has been extended to functionals in \cite{TiaHon95}: We emphasize that we use here the notation $\X=C(\overline{\Omega_\X})$ for consistency reasons, although this is not a Hilbert space as assumed in Notation \autoref{not:xy}.
\begin{theorem}[Approximation of functionals (Theorem 4, \cite{TiaHon95})]\label{th:functional_apx_cont}
	Let $\sigma$ be a sigmoid activation function as defined in \autoref{def:sigm_act}. Let 
	$\Omega_\X \subseteq \R^m$ be a bounded domain.
	Moreover, let 
	\begin{equation*}
		F: \dom{F} \subseteq C(\overline{\Omega_\X}) \to \R
	\end{equation*}
	be a continuous functional with respect to the $\norm{\cdot}_{L^\infty(\Omega_\X)}$-topology, and suppose that the domain of $F$, $\dom{F}$, is compact in $C(\overline{\Omega_\X})$. Then, for every $\ve >0$ and every $\x \in \dom{F}$, there exist numbers
	$$K := K(\ve), \; L := L(\ve) \in \N,$$
	sampling points 
	$$\bs_1 := \bs_1(\ve), \dots, \bs_{L}:=\bs_L(\ve) \in \overline{\Omega_\X}\,,$$ 
	and real numbers
	$$ w_{k,l} := w_{k,l}(\ve), \theta_k := \theta_k(\ve) \text{ for } k=1,\dots, K, l=1,\dots, L,$$ 
	(all of the above coefficients are independent of $\x$)
	and coefficients, which are dependent of $\x$,  
	$$ C_k := C_k(\ve,\x),$$
	such that the functional
	\begin{equation}\label{eq:functional_finite}
		\x \in \dom{F} \mapsto F_\n[\x] :=  \sum_{k=1}^{K} C_k \sigma\left(\sum_{l=1}^{L} w_{k,l}\x(\bs_l) + \theta_k\right)
	\end{equation} 
	satisfies 
	\begin{equation} \label{eq:functional_approx}
		\abs{F[\x]-F_\n[\x]} \leq \ve \,.
	\end{equation}
	We abbreviate the coefficients of the functional as follows:
	\begin{equation*}
		\n = \abs{\mathcal{T}_\n} = K(L+2) + Lm
		\text{ and }
		\mathcal{T}_\n:= \begin{pmatrix}  \underbrace{ C_{k}}_{\in \R} & \underbrace{ w_{k,l}}_{\in \R} & \underbrace{ \theta_k}_{\in \R} & \underbrace{ \bs_l}_{\in \R^m} &\end{pmatrix}_{\tiny \begin{array}{c} k=1,\dots,K\\ l=1,\dots,L\end{array}}\;.
	\end{equation*}
\end{theorem}

In the next step, we review approximation of $F$ with neural operators:
\subsection{Approximation of operators}
\emph{Neural operators} are novel tools for approximating \emph{nonlinear} operators (see \cite{LuJinPanZhaKar21}). The approach extends neural network approximations of continuous functions ${\tt u}: \overline{\Omega} \to \R$ and functionals $F: \dom{F} \subseteq C(\overline{\Omega_\X}) \to \R$ to operators $F: \dom{F} \subseteq C(\overline{\Omega_\X}) \to L^2(\Omega_\Y)$.

\begin{definition}[Neural operator] \label{de:deeponet} 
	Let $\sigma$ be a bounded sigmoid activation function as defined in \autoref{def:sigm_act}. 
	Let $\Omega_\X \subseteq \R^m$ and 
	$\Omega_\Y \subseteq \R^n$ be bounded domains, respectively.
	The operator $F_\n$, defined by
\begin{equation} \label{eq:deeponet}
	\begin{aligned}
		F_\n: \dom{F} \subseteq C(\overline{\Omega_\X}) &\to L^2(\Omega_\Y)\,,\\
		{\tt x} &\mapsto \left( \bt \in \Omega_\Y \mapsto \sum_{j=1}^{J} \sum_{k=1}^{K} \alpha_{j,k} \sigma \left( \sum_{l=1}^{L} w_{j,k,l} {\tt x}(\bs_l)+\theta_{j,k}\right) \sigma (\bw_j^T \bt +\zeta_j) \right)
	\end{aligned}
\end{equation}
is called a \emph{neural operator}. 
The coefficients are summarized in a vector
\begin{equation}\label{eq:D}
\mathcal{T}_\n = \begin{pmatrix} \underbrace{ \alpha_{j,k}}_{\in \R} & \underbrace{ w_{j,k,l}}_{\in \R} & \underbrace{ \bw_j}_{\in \R^n} & \underbrace{ \theta_{j,k}}_{\in \R} & \underbrace{ \bs_l}_{\in \Omega_\X \subseteq \R^m} 
	& \underbrace{ \zeta_j}_{\in \R} &\end{pmatrix}_{\tiny \begin{array}{c} j=1,\dots,J\\ k=1,\dots,K\\ l=1,\dots,L \end{array}} \;.
	\end{equation}
	with 
	\begin{equation*}
	\n = \abs{\mathcal{T}_\n} = 
	J(K(L + 2) + n + 1) + Lm
	\end{equation*}
\end{definition}
Again, in comparison with \autoref{th:functional apx}, we have a significant freedom in choosing the coefficients. In fact, many of them are customized (see \autoref{ex:customized}).

\begin{theorem}[Approximation of operators (Theorem 5, \cite{TiaHon95})]\label{th:operatorapxc} 
	Let $\sigma$ be a  sigmoid activation function. Let $\Omega_\X, \Omega_\Y$ be bounded domains in $\R^m, \R^n$, respectively, where $\Omega_\X$ is bounded with piecewise $C^1$ boundary. Moreover, let 
	\begin{align*}
		F: \dom{F} \subseteq C(\overline{\Omega_\X}) &\to L^2(\Omega_\Y)
	\end{align*}
	be a continuous operator with respect to the $\norm{\cdot}_{L^\infty(\Omega_\X)}$-topology on the domain of $F$, $\dom{F}$, and assume that $\dom{F}$ is compact in $C(\overline{\Omega_\X})$. 	
	Then, for every $\ve >0$ and every $\x \in \dom{F}$ there exist positive integers 
	$$J := J(\ve), \; K := K(\ve),\; L := L(\ve) \in \N\,,$$ 
	and for all $j=1,\dots,J$, $k=1,\dots,K$, $l=1,\dots, L$, there exists real parameters
	$$ w_{j,k,l}:= w_{j,k,l}(\ve), \; \theta_{j,k} := \theta_{j,k}(\ve),\;
	\zeta_{j}:= \zeta_j(\ve) \in \R$$ 
	and vectors $$\bw_j:=\bw_j(\ve) \in \R^n$$ as well as sampling points 
	$$\bs_l:=\bs_l(\ve) \in \Omega_\X\,,$$
	(all of them are independent of $\x$). Moreover, there exist coefficients (depending on $\x$ and $\ve$) 
	$$\alpha_{j,k} := \alpha_{j,k}(\x,\ve)$$ such that $F_\n$ from \autoref{eq:deeponet} satisfies
	\begin{equation*}
		\norms{F[\x]-F_\n[\x]}_{L^2(\Omega_\Y)} \leq \ve\;.
	\end{equation*}
\end{theorem}

\subsection{Training of neural operators}
Efficient computing coefficients of a neural operator, also referred to training of a network, is essential for accurate approximation of the operator $F$. The training problem is formulated as follows:
\begin{definition}[Training of neural operators] 
The coefficients $\mathcal{T}_\n$ of a neural operator are determined through supervised training samples 
\begin{equation} \label{eq:S}
	\mathcal{S}_\no := \set{(\hat{\x}^{(\ell)},\hat{\y}^{(\ell)}) : \hat{\y}^{(\ell)} = \op{\hat{\x}^{(\ell)}} , \ell=0,1,\ldots,\no}\,,
\end{equation}
such that
\begin{equation} \label{eq:learn}
	\begin{aligned}
		\hat{\y}^{(\ell)}(\bt_\rho) 
		&= \sum_{j=1}^{J} \sum_{k=1}^{K} \alpha_{j,k} \sigma \left( \sum_{l=1}^{L} w_{j,k,l} \hat{\x}^{(\ell)}(\bs_l)+\theta_{j,k}\right) \sigma (\bw_j^T \bt_\rho +\zeta_j) = F_\n[\hat{\x}^{(\ell)}](\bt_\rho)\;.
	\end{aligned}
\end{equation}
Here $\set{\bt_\rho : \rho=1,\ldots,Q}$ are sampling points in $\Omega_\Y$, which we assume to be given. 
We also use the training data centered at $(\hat{\x}^{(0)},\hat{\y}^{(0)})$, which is given by 
\begin{equation} \label{eq:S0}
	\mathcal{S}_\no^{(0)} := \set{(\x^{(\ell)},\y^{(\ell)}) := (\hat{\x}^{(\ell)},\hat{\y}^{(\ell)}) -  (\hat{\x}^{(0)},\hat{\y}^{(0)}): \ell=1,\ldots,\no}\;.
\end{equation}
\end{definition}

We summarize some observations on neural operators, which guide the further paper:
\begin{enumerate}
	\item The \emph{universal approximation theorem for operators} from \cite{TiaHon95,LuJinPanZhaKar21} guarantees the existence of coefficients 
	$\mathcal{T}_\n$ such that the corresponding neural operator $F_\n$ uniformly approximates the operator $F$ on a compact subset of continuous images from $\X$. However, this analysis is not quantitative in terms of topologies, which are required for inverse problems applications (see \autoref{ex: darcyflow} and \autoref{ex:2point}).
	\item Neural operators (see also \autoref{de:deeponet}) can be applied to elements of $\X$ that allow for point evaluations (see \cite{WanWanPer21}), which limits the applicability. For instance, in \autoref{ex:2point}, the domain of $F$ consists of $L^2$ functions, and therefore neural operators cannot be applied directly. The generalization to $L^2$ function spaces is considered in \autoref{se:generalization}). In particular, approximation properties of surrogate operators must be analyzed in Lebesgue and Sobolev spaces (see \autoref{se:examples}). Sobolev spaces are somewhat easier to handle because they can often be compactly embedded into spaces of continuous functions (see \autoref{sec:sobolevproof}). This fact is used for analyzing the $a$-example (see \autoref{ex: darcyflow}) with a neural operator, which surrogates the nonlinear operator $F$. 
    \item Training neural operators is highly nonlinear and therefore a complex computational problem. Efficient implementation requires priors for determining coefficients $\mathcal{T}_\n$ and their counts $J, K, L$ (see \autoref{eq:D}) before initiating the actual training process using a gradient descent algorithm. 
    \item \autoref{eq:learn} does not need to hold exactly but only with up to some tolerance $\ve$. According to the theoretical results (see \autoref{th:operatorapxc}), the coefficients $\alpha_{j,k}$ can be chosen in dependency of $\hat{\x}^{(\ell)}$, $\ell = 1,\ldots,\no$, which is not reflected in the formula of \autoref{eq:learn}. However, in \autoref{eq:learn}, we did not specify the amount of parameters $J$, $K$, $L$. The calculation in later section actually indicate that they should be dependent on $\no$. See \autoref{eq:nldeepm} below.
\end{enumerate}
We begin with studying the third item and determine a good set of prior coefficients for $\mathcal{T}_\n$. We see below that several parameters of the neural operator can in fact be predetermined (customized in the language of neural networks). We recall that the dependency of parameters described in \autoref{th:function apx}, \autoref{th:functional_apx_cont}, and \autoref{th:operatorapxc} reveals the potential of customizing parameters, which significantly simplifies the training process.

\section{Priors for training of neural operators} \label{se:orthogonal}
We propose the following two-step strategy to determine prior coefficients $\mathcal{T}_\n$ (see \autoref{eq:D}) of the surrogate operator $F_\n$ (see \autoref{eq:deeponet}): 
	\begin{enumerate}
	\item We approximate the operator $F$ by a \emph{linearization} (see $F_\no^\#$ in \autoref{eq:nldeep}), where
	\item the expanding functions are \emph{non-linearly} approximated by neural network functions, leading to $\tilde{F}_\n$, which is a nonlinear operator with respect to the images $\x$ (see \autoref{eq:nldeep}).
\end{enumerate}
The expansion of the linearization allows us to determine, in particular, an adequate number of coefficients $J$, $K$ and $L$ of the neural operator $F_\n$ (see \autoref{eq:deeponet}).

\subsection{Neural operators for linear operator regression} \label{ss:lorl}
In this subsection, we assume that the operator $F$ in \autoref{eq:op} is linear. Later this will be a linearization of a nonlinear operator. In order to learn the linear finite dimensional operator $F_\no^\#$, we use a strategy developed in \cite{AspFriSch24_preprint}, which consists in orthonormalizing the centered training images $\x^{(\ell)}$, $\ell=1,\dots,\no$ (see \autoref{eq:S0}) with Gram-Schmidt and computing the according data. Gram-Schmidt requires that the training images $\x^{(\ell)}$, $\ell = 1,\ldots,\no$, are linearly independent, which we assume to hold. Additionally, we assume that the linear operator $F$ has a trivial nullspace. Consequently, the training data $\y^{(\ell)}$ are also linearly independent. More sophisticated algorithms have been investigated in \cite{HanSch26}. In particular, the techniques developed there do not require the assumptions that the training images $\x^{(\ell)}$, $\ell = 1,\ldots,\no$, are linearly independent and that $F$ has a trivial nullspace. 

In the following, we denote by $\set{\ulx^{(\ell)}: \ell=1,\dots,\no}$ an orthonormalized family obtained from  $\set{\x^{(\ell)}: \ell=1,\dots,\no}$,  which have the same span and we denote by $\uly^{(\ell)} = F\ulx^{(\ell)}$ the according family of data. We emphasize that under the above assumptions (linear independence and trivial nullspace of $F$), there exist explicit formulas for $\uly^{(\ell)}$ (see \cite{AspKorSch20,AspBanOekSch20}).
Let 
	\begin{equation} \label{eq:sample}
		\begin{aligned}
			\X_\no &:= \text{span}\set{\ulx^{(\ell)}:\ell = 1,\dots,{\no}} = \text{span}\set{\x^{(\ell)}:\ell = 1,\dots,{\no}}
			\text{ and} \\ \quad \Y_\no &:= \text{span}\set{\uly^{(\ell)}:\ell = 1,\dots,{\no}}= \text{span}\set{\y^{(\ell)}:\ell = 1,\dots,{\no}},
		\end{aligned}
	\end{equation}
	which are finite dimensional subsets of $\X$ and $\Y$, respectively. For every image $\x \in \X_\no$, we have the basis expansion
	\begin{equation*}
		\x = \sum_{\ell=1}^\no \inner{\x}{\ulx^{(\ell)}}_\X \ulx^{(\ell)},
	\end{equation*}
	and from the linearity of $F$, we obtain
	\begin{equation*} 
		\y = F\x = \sum_{\ell=1}^\no \inner{\x}{\ulx^{(\ell)}}_\X F \ulx^{(\ell)} = 
		\sum_{\ell=1}^\no \inner{\x}{\ulx^{(\ell)}}_\X \uly^{(\ell)}\; \in \Y_\no\;.
	\end{equation*}
    Now, we define an operator
	\begin{equation} \label{eq:expansion}
		\begin{aligned}
			F_\no^\#: \X &\to \Y\;\\
			\x &\mapsto \sum_{\ell=1}^\no \inner{\x}{\ulx^{(\ell)}}_\X \uly^{(\ell)}.
		\end{aligned}
	\end{equation}
	Note that $F_\no^\#$ and $F$ coincide on $\X_\no$, that is, $F_\no^\#|_{\X_\no} = F|_{\X_\no}$. Moreover, $F_\no^\# = 0$ on $\X_\no^\perp$.
	The operator $F_\no^\#$ is expressed as a sum of products of functionals $\x \mapsto \inner{\x}{\ulx^{(\ell)}}_\X$, which can be approximated by networks as in \autoref{th:functional_apx_cont}, and functions
	$\bt \mapsto \uly^{(\ell)}(\bt)$, which can be approximated using standard neural network as in \autoref{th:function apx}.
	
	By substituting the neural network approximations of functionals (see \autoref{th:functional apx}) and functions (see \autoref{th:function apx}) into \autoref{eq:expansion}, and 
	under the assumption that every function $\uly^{(\ell)}, \ell = 1,\dots,\no$, is continuous, we obtain that
	\begin{equation} \label{eq:nldeep}
		\begin{aligned}
			F_\no^\# \x(\bt) &= 
			\sum_{\ell=1}^\no \inner{\x}{\ulx^{(\ell)}}_\X \uly^{(\ell)}(\bt) \\
			&\approx \sum_{\ell=1}^\no \left(\sum_{k=1}^{K(\ell)} C_k^{(\ell)}\sigma \left(\sum_{l=1}^{L}w_{k,l}^{(\ell)} \x(\bs_l)+\theta_{k}^{(\ell)}\right)\right)\left(\sum_{j=1}^{J(\ell)}c_j^{(\ell)}\sigma\left(\bw_j ^{(\ell)T}\bt + \zeta_j^{(\ell)}\right)\right)\\
			&= \sum_{\ell=1}^\no \sum_{j=1}^{J(\ell)} \sum_{k=1}^{K(\ell)}\alpha_{j,k}^{(\ell)}\sigma \left(\sum_{l=1}^{L}w_{k,l}^{(\ell)}\x(\bs_l) + \theta_k^{(\ell)}\right) \sigma\left(\bw_j^{(\ell) T}\bt+\zeta_j^{(\ell)}\right):=\tilde{F}_\n[\x](\bt)
		\end{aligned}
	\end{equation}
for all $\x \in \X$ and $\bt \in \Omega_\Y$. Note that we already did a customization and defined the sampling points $\bs_l$, $l=1,\ldots,L(\ell)$, independent of $\ell=1,\ldots,\no$. Due to the nonlinearity of the activation function $\sigma$, $\tilde{F}_\n$ is a nonlinear operator, although we approximated the linear operator $F_\no^\#$. On the negative side, this approximation destroys the simplicity of a linear system, increases training complexity and makes the error analysis more complicated. On the positive side, a nonlinear network can approximate both linear and nonlinear operators. Note that, for nonlinear operators we use brackets, $[\cdot]$, for operator evaluations, which are left out for linear operators. In comparison to \autoref{eq:deeponet}, there is an additional summation over $\ell = 1, \ldots, \no$ in \autoref{eq:nldeep}. However, with an inductive mapping 
$$ \tilde{j} = j \ell, \tilde{k} = k \ell \text{ with } \ell=1,\ldots,\no \text{ and } k=1,\ldots,K, j=1,\ldots,J,$$ 
		we define
    \begin{equation*}
    	\tilde{\alpha}_{\tilde{j},\tilde{k}} = \alpha_{j,k}^{(\ell)}, \tilde{w}_{\tilde{k},l} = w_{k,l}^{(\ell)}, \tilde{\theta}_{\tilde{k}} = \theta_k^{(\ell)},
    	\tilde{\bw}_{\tilde{j}} = \vw_j^{(\ell)} \text{ and } \tilde{\zeta}_{\tilde{j}} = \zeta_j^{(\ell)}
    \end{equation*}
     and we obtain from \autoref{eq:nldeep} that
	\begin{equation} \label{eq:nldeepm}
		\begin{aligned}
			\tilde{F}_\n[\x](\bt) 
			&=\sum_{\ell=1}^\no \sum_{j=1}^{J} \sum_{k=1}^{K}\alpha_{j,k}^{(\ell)}\sigma \left(\sum_{l=1}^{L}w_{k,l}^{(\ell)}\x(\bs_l) + \theta_k^{(\ell)}\right) \sigma\left(\bw_j^{(\ell) T}\bt+\zeta_j^{(\ell)}\right)\\
			&=\sum_{\tilde{j}=1}^{\tilde{J}} \sum_{\tilde{k}=1}^{\tilde{K}}\tilde{\alpha}_{\tilde{j},\tilde{k}}\sigma \left(\sum_{l=1}^{L}\tilde{w}_{\tilde{k},l}\x(\bs_l) + \tilde{\theta}_{\tilde{k}}\right) \sigma\left(\tilde{\bw}_{\tilde{j}} ^{T}\bt+\tilde{\zeta}_{\tilde{j}}\right)\,,
		\end{aligned}
	\end{equation}
which corresponds with the form of $F_\n$ from \autoref{eq:deeponet}.

\begin{remark}
Let $\{(\ulx^{(\ell)},\uly^{(\ell)})\}_{\ell=1}^{\no}$, $\no \in \N$, be centered training pairs. The structure of the operator $\tilde{F}_\n$ in \autoref{eq:nldeepm} explains the terminology from \cite{LiKovAziLiuBha20_report} of neural operators as defined in \autoref{eq:deeponet}: Using the functional approximation of $\x \mapsto \inner{\x}{\ulx^{(\ell)}}_\X$ as the \emph{branch network} and the function approximation of $\bt \mapsto \uly^{(\ell)}(\bt)$ as the \emph{trunk network}, we find the network architecture of the neural operator $F_\n$ as defined in \autoref{eq:deeponet}. Therefore, we no longer distinguish between them notation wise. Moreover, this construction shows that the number of coefficients $\n$ depends on the number of training samples $\no$, which is consistent with the approximation result \autoref{th:operatorapxc}.
\end{remark}

	Now, we present an approximation result for linear operators with neural operators:
	\begin{corollary} \label{cor:rate}
		Let $F:\X \to \Y$ be a linear operator, and let $\tilde{F}_\n$, $F_\no^\#$ be as defined in \autoref{eq:nldeep}. Assume that there exist rate functions $q,r :\N \to (0,\infty)$ such that for all $\x \in \mathcal{B}_{\tt r}(\xdag)$ and $\ell =1,\ldots,\no$, the following estimates hold independently of $\ell$ (which means that the following estimates are independent of the training sample):
		\begin{equation}\label{eq:qn}
			\abs{\inner{\x}{\ulx^{(\ell)}}_\X -  \sum_{k=1}^{K} C_k^{(\ell)}\sigma \left(\sum_{l=1}^{L}w_{k,l}^{(\ell)} \x(\bs_l)+\theta_{k}^{(\ell)}\right)} \leq q(\no)
		\end{equation}
		and 
		\begin{equation}\label{eq:rn}
			\norm{\bt \mapsto \uly^{(\ell)}(\bt)- \sum_{j=1}^{J}c_j^{(\ell)}\sigma\left(\bw_j^{(\ell)T}\bt + \zeta_j^{(\ell)}\right)}_\Y \leq r(\no)\;.
		\end{equation}
		Then, it holds that
		\begin{equation} \label{eq:estimate}
			\begin{aligned}
				\norms{F\x-\tilde{F}_\n[\x]}_\Y &\leq \norms{F\x-F_\no^\#\x}_\Y + \norms{F_\no^\#\x-\tilde{F}_\n[\x]}_\Y 
				&\leq \underbrace{\norms{(I-P_\no)F}}_{=:\nu_\no} \underbrace{\norm{\x}_\X}_{\leq \norm{\xdag}_\X +{\tt r}} +  \no q(\no) r(\no) \;,
			\end{aligned}
		\end{equation}
		where $P_{\no}$ is the orthonormal projector onto $\Y_\no$.		
	\end{corollary}

In the following, we summarize some remarks about the estimate \autoref{eq:estimate}:
	\begin{remark}
	   \begin{itemize} 
	   	\item $\nu_\no$ denotes the operator norm of $(I-P_\no)F$, which converges to $0$ when $F$ is compact.
       \item We assume that the left hand sides of \autoref{eq:qn} and \autoref{eq:rn} are of the order $K^{-p}$ and $J^{-q}$ if $\x$ and $\y$ satisfy some smoothness assumptions \cite{Bar93} and the sampling is fine enough. We choose for the same of simplicity 
        \begin{equation*}
        	K = \mathcal{O}(\no) \text{ and } J = \mathcal{O}(\no),
        \end{equation*}
        such that 
        \begin{equation} \label{eq:no}
        	\no q(\no) r(\no) = \mathcal{O} (\no^{1-p-q})\;.
        \end{equation}
        In total, we therefore have an estimate 
        \begin{equation}\label{eq:final}
        	\norms{F\x-\tilde{F}_\n[\x]}_\Y \leq \mathcal{O}(\max \set{\nu_\no,\no^{1-p-q}})\;.
        \end{equation}
	   \end{itemize}
	\end{remark}
\autoref{eq:estimate} is applicable to ill-posed, infinite dimensional problems, when the operator $F$ is compact. For a well-posed problem, where $F$ is not compact, $\norm{(I-P_\no)F}_{\X \to \Y}$ does, in general, not converge to $0$.
	
\begin{example}\label{ex:customized} The purpose of this example is to show that many of the parameters in neural operators, such as introduced in \autoref{eq:deeponet} and \autoref{eq:nldeepm}, can be predetermined. Some parameters can be found, for instance, from classical integration formulas, as we show in the next example. Moreover, integration formulas also provide approximation error estimates.
	\begin{itemize}
		\item Let $\X = L^2(0,1)$ and let $t_k = \frac{k}{K}$, $k=0,1,\ldots,K$. Moreover, assume that $\x$ and $\ulx^{(\ell)}$ are twice differentiable, then:
		\begin{equation*}
			\begin{aligned}
			\inner{\x}{\ulx^{(\ell)}}_{L^2} &= \int_0^1 \x(t) \ulx^{(\ell)}(t) \mathrm{d}t \\
			&= 
			\frac{1}{2K} \x(0) \ulx^{(\ell)}(0) + \frac{1}{K} \sum_{k=1}^{K-1} \x(t_k) \ulx^{(\ell)}(t_k) + \frac{1}{2K} \x(1) \ulx^{(\ell)}(1) + \mathcal{O}(K^{-2})\;.
			\end{aligned}
		\end{equation*}
		To satisfy \autoref{eq:qn}, we aim to solve for $k=0,1,\ldots,K$ and $\ell=1,2,\ldots,N$ the system of equations:
	    \begin{equation} \label{eq:expl}
	       \x(t_k) \ulx^{(\ell)}(t_k)	= 
	       \sigma \left(\sum_{l=1}^{L}w_{k,l}^{(\ell)}\x(t_l) + \theta_k^{(\ell)}\right) + \mathcal{O}(K^{-3}) \text{ for } k=0,\ldots,K. 
	    \end{equation}
	    Then, the functional $\x \mapsto \inner{\x}{\ulx^{(\ell)}}_{L^2}$ is approximated in the form required in \autoref{eq:qn},  with an accuracy of $\mathcal{O}(K^{-3})$.
	    Depending on the activation function $\sigma$, this is a complicated nonlinear equation. If, however, $\sigma$ is a ReLU-network, that is, 
	    \begin{equation*}
	    	\sigma(t) = \left\{ \begin{array}{rl} t & \text{ for } t >0 \\
	    	0 & \text{ otherwise} \end{array}\right.
	    \end{equation*}	    
	     and $\x(t_k)$ and $\ulx^{(\ell)}(t_k)$ are positive (which we have to assume, if we want to represent $\x$ and $\ulx^{(\ell)}$ with sigmoid functions), then, because $\sigma(t)=t$ for positive values of $t$, we obtain with a choice $L=1$ that
        \begin{equation} \label{eq:expl2}
        	\x(t_k) \ulx^{(\ell)}(t_k) = w_{k}^{(\ell)} \x(t_k) \;,
        \end{equation} 
        which carries over the sampling of the training functions to the weights.
        Moreover, comparing with \autoref{eq:nldeep}, the coefficients $C_k^{(\ell)}$ are determined by
        \begin{equation*}
        	C_k^{(\ell)} = \begin{cases} \frac{1}{2K} &\text{ for } k = 0 \text{ and } K\\
        		\frac{1}{K} &\text{ otherwise. }
        	\end{cases}
        \end{equation*}
       These coefficients are therefore given by the trapezoidal rule.
      If we choose the number of discretization points in the trapezoidal rule to match the number of training data points (i.e., if $K=\no$), then we obtain the rate for $q_\no$ defined in \autoref{eq:qn}
        \begin{equation*}
        	q(\no) = \mathcal{O}(\no^{-2})\;.
        \end{equation*} 
        \item Common quantitative estimates for $r(\no)$ are of the root of the number of sums, that is of $\mathcal{O}(J^{-1/2})$ (see \cite{Bar93}). With such an estimate we get from \autoref{eq:estimate} that
        \begin{equation} \label{eq:estNqr}
        	\no q(\no) r(\no) = \mathcal{O}(\no^{-1} J^{-1/2})\;.
        \end{equation}
        By choosing $J \sim \no$ we get the rate 
        \begin{equation*}
        	\no q(\no) r(\no) = \mathcal{O}(\no^{-\frac{3}{2}})
        \end{equation*}
        for the 2nd term in \autoref{eq:estimate}. There is still room for improvements of the convergence rate. One may expect faster convergence rates depending on the smoothness of the function to be approximated, $\uly^{(\ell)}$, as observed in other approximation methods. The rate $\mathcal{O}(J^{-1/2})$ arises in Monte Carlo-type approximations based on random sampling, which mainly driven by randomness and does not account for the regularity of the target functions $\uly^{(\ell)}$. However, if the function has high regularity, for example, if $\uly^{(\ell)} \in H^s(\Omega_\Y)$ for $s > \frac{n}{2}$, where $n$ is the dimension of $\Omega_\Y$, then neural networks approximation can achieve a faster convergence rate of $\mathcal{O}(J^{-s/n})$  (see \cite{DevHanPet21}). Deterministic methods such as spectral methods or polynomial type approximations can achieve error rate of $\mathcal{O}(J^{-k})$ if $\uly^{(\ell)} \in C^k(\overline{ \Omega_\Y})$, and even exponential decay rate if $\uly^{(\ell)}$ is analytic (see \cite{Kat76}, \cite{Dun30}).
	\end{itemize}
	
	This example shows that for a neural operator, the structure of the coefficients can be determined from known approximation methods.
\end{example}
Now, we come to the main theorem of this subsection:
\begin{theorem}[$F$ linear] \label{th:linear} Let  $\tilde{F}_\n$ as defined in \autoref{de:tildeF} and 
let $\x_\n^{\alpha\delta\eta}$ be an approximate minimizer of the functional
\begin{equation} \label{eq:discTik_st}
	\tilde{\Tik}_{\n}^{\alpha\delta}[\x] := \norms{\tilde{F}_\n[\x]-\yd}_\Y^2 + \alpha \norms{\x}^2_{\X}\,,
\end{equation}
with tolerance $\eta$ (see \autoref{eq:discTik}) over $\mathcal{D}(F)$. As stated in the introduction, in linear regularization theory, one typically uses $\x^{(0)}=0$, a setting, which we follow in this paper.
Let ${\tt r} > 0$ such that $\mathcal{B}_{\tt r}(\xdag) \subseteq \mathcal{D}(F)$, where $\xdag$ denotes the minimum solution defined in \autoref{eq:argmin} with $\x^{(0)}=0$. We assume that the following \emph{source condition} holds
\begin{equation} \label{eq:source_linear}
	\xdag = F^* \omega\;.
\end{equation}
Taking into account the notation from \autoref{cor:rate} (in particular recall that $\nu_\no = \norm{(I-P_\no)F}$),
we choose 

\begin{equation*}
	\alpha \sim \max\{\delta,\nu_\no ( \norm{\xdag}_\Y + r) + \no q(\no)r(\no)\} \; \text{ and } \; \eta \sim \alpha^2 \;.
\end{equation*} 
Then
\begin{equation*}
	\norms{\x_\n^{\alpha\delta\eta}-\xdag}_{L^2} =  \mathcal{O}(\sqrt{\delta} + \sqrt{\nu_\no + \no q(\no)r(\no)})\;.
\end{equation*} 
\end{theorem}
The proof of this result is analogous to the proof of \cite[Theorem 2.3a]{NeuSch90}. 
	
\begin{remark} 
The term $\nu_\no (\norm{\x}_\X+{\tt r})+\no q(\no)r(\no)$ is a typical estimate for an operator approximation (see \autoref{eq:unifrom}) in regularization theory when operator perturbations are considered (see \cite{Gro84,EngKunNeu89,Neu89}). 
The estimate consists of two components, where one arises from the discretization of the elements of $\x$ (leading to the estimate with $\nu_\no$) and the second one, which approximates the functional $\x \to \inner{\x}{\ulx^{(\ell)}}_\X$ and $\uly^{(\ell)}$, respectively.
In the context of finite element approximations, it has been shown (see \cite{NeuSch90}) that when $\Y_\no$ is the space of linear splines, one has $\norm{F[\x]-F_\n[\x]}_\Y = \mathcal{O}(\n^{-2})$ in a neighborhood of $\xdag$ (note that for finite element approximations we have $\n \sim \no$)
in \autoref{ex: darcyflow} and \autoref{ex:2point}. $q(\no)$ (see \autoref{eq:qn}) is the approximation error of the discretization of an inner product (typically using a quadrature formula) and $r(\no)$ (see \autoref{eq:rn}) represents the approximation error of the expert data.
\end{remark}

We have shown how to explicitly calculate a neural operator that approximates a linear operator $F$. If $F$ is nonlinear, we can use the aforementioned strategy as motivation.

\subsection{Neural operators for nonlinear operator regression}
We generalize the idea from \autoref{ss:lorl} to nonlinear operators as follows:
\begin{definition} \label{de:tildeF}
\begin{itemize}
	\item We first linearize the operator $F$ around $\hat{\x}^{(0)}$, which is the $1$st component of the $0$th element of $\mathcal{S}_\no$ in  \autoref{eq:S}. That is, we use the following approximation:
	\begin{equation}\label{eq:linear}
		F[\hat{\x}] - F[\hat{\x}^{(0)}] \approx F'[\hat{\x}^{(0)}](\hat{\x}-\hat{\x}^{(0)}) =: F'[\hat{\x}^{(0)}]\x\;.
	\end{equation}
    \item Then, we approximate the linear operator $F'[\hat{\x}^{(0)}]$ with an operator $F_\no^\#$ (see \autoref{ss:lorl} above), which is learned from the shifted training data 
    \begin{equation} \label{eq:S_G}
    	 \mathcal{S}_{F'[\hat{\x}^{(0)}]}^{(0)} := \{(\underbrace{\hat{\x}^{(\ell)}-\hat{\x}^{(0)}}_{=:\x^{(\ell)}},\underbrace{\hat{\y}^{(\ell)}-\hat{\y}^{(0)}}_{=:\y^{(\ell)}}) : \ell=1,\dots,\no \}\;.
    \end{equation}
    The operator $F_\no^\#$ is determined after orthonormalization of the set $\set{\x^{(\ell)}: \ell =1,\ldots,\no}$:
    \begin{equation} \label{eq:lin}
    	F'[\hat{\x}^{(0)}] \x \approx F_\no^\# \x := \sum_{\ell=1}^\no \inner{\x}{\ulx^{(\ell)}}_\X \uly^{(\ell)}.
    \end{equation}
    Here, $\ulx^{(\ell)}$ is an orthonormal basis computed from $\x^{(\ell)}$, $\ell=1,\dots,\no$, and $\uly^{(\ell)}$ is computed with the iterative algorithm from \cite{AspFriSch24_preprint}, which only requires the expert information $ \mathcal{S}_{F'[\hat{\x}^{(0)}]}^{(0)}$. Note, however, that in general, for a nonlinear operator $F$, $\uly^{(\ell)} +\y^{(0)} \neq F[\ulx^{(\ell)}+\x^{(0)}]$. 
    Nevertheless, because of \autoref{eq:linear}, we expect that $F_\no^\#$ still provides a good approximation of $F$ in a neighborhood of $\xdag$, provided that the elements of $ \mathcal{S}_{F'[\hat{\x}^{(0)}]}^{(0)}$ are also in a neighborhood of $\xdag$.
   \item Finally, we approximate $F_\no^\#$ as follows (analogously as in \autoref{eq:nldeep}, where $F$ is linear):
          \begin{itemize}
          	\item The family of functionals $\x \mapsto \inner{\x}{\ulx^{(\ell)}}_\X$, $\ell=1,\dots,\no$,
          	      is approximated by neural functional approximations from \cite{CheChe93} (see \autoref{eq:qn}). 
          	\item The family of functions $\vec{t} \in \Omega_\Y \mapsto \uly^{(\ell)}(\vec{t})$ is approximated by neural networks from \cite{Cyb89} (see \autoref{eq:rn}).  
          	\item Using these two approximations, we obtain the approximation $\tilde{F}_\n$ (see \autoref{eq:deeponet}).
          \end{itemize} 
          Recall, that $\n$ denotes the number of coefficients of the neural operator and $\no$ denotes the number of training samples. In an analysis they need to be balanced.
\end{itemize}
\end{definition}
At this point it is convenient to summarize the different operators used in \autoref{ta:operators}.
\begin{table}[h!]
   \centering \footnotesize
			\renewcommand{\arraystretch}{1.2} 
	\begin{tabular}{|c|c|c|c|c|c|}
		\hline
		Operator &  Linear/Nonlinear & Reference & Parameter & \multicolumn{2}{c|}{Approximation rate} \\
		\hline
		$F$ & NL & \autoref{eq:op} & -- & \multicolumn{2}{c|}{--} \\
		\hline
		\multirow{2}{*}{$F^{\text{FE}}_\n$} & \multirow{2}{*}{NL} & \autoref{eq:fea}, & \multirow{2}{*}{No. of finite elements} & \multicolumn{2}{c|}{\multirow{2}{*}{$n^{-2}$}}  \\
		& & \autoref{eq:fec} & & \multicolumn{2}{c|}{ } \\
		\hline
		$F_\n$ & NL & \autoref{eq:deeponet} & $\n = \abs{\mathcal{T}_\n}$ \autoref{eq:D} & \multicolumn{2}{c|}{$\rho_\n$} \\
		\hline
		$F_{\no}^\#$ & L & \autoref{eq:lin} & $\no = \abs{\mathcal{S}_{\no}}$ \autoref{eq:S} & \multicolumn{2}{c|}{$\nu_{\no}$} \\
		\hline
		\multirow{2}{*}{$\tilde{F}_\n$} & \multirow{2}{*}{NL} & \multirow{2}{*}{\autoref{eq:nldeep}} & \multirow{2}{*}{$F_{\no}^\# \approx \tilde{F}_\n$}
		& $F$ linear
		& $F$ non-linear
		\\ \cline{5-6}
		 &  & & 
		& $\nu_{\no}(\norm{\xdag}_{\X}+r)+\no\,q(\no)r(\no)$
		& $\rho_\n$
		\\
		\hline
	\end{tabular}
	\caption{\label{ta:operators} The different operators used in this paper. Only $F_\no^\#$ is a linear operator. We emphasize that the index $\n$ in $F_\n$ and $\tilde{F}_\n$ refers to the number of coefficients, $\n=\abs{\mathcal{T}_\n}$ while the index in $F_\no^\#$ refers to the number of training samples. The operator $\tilde{F}_\n$ is a neural operator obtained algorithmically by orthonormalizing the training images and evaluating the corresponding data. This strategy is exactly implementable for linear inverse problems, but not for nonlinear inverse problems, where the evaluation of the operator at training images only delivers an approximation. The classical finite element operator approximation is denoted by $F^{\text{FE}}_\n$. In this cases, $\n$ represents the number of mesh elements, and $\frac{1}{\n}$ is the mesh size. }
\end{table}

\section{Neural functionals and operators in Sobolev spaces} \label{sec:basisrep}
As we have seen in \autoref{ex: darcyflow} and \autoref{ex:2point}, in inverse problems, the forward operators $F$ is often defined on subsets of Sobolev or Lebesgue spaces. Such spaces aligns very well with quadratic Tikhonov regularization in Hilbert spaces.

The ultimate goal of this section is to show that the operator $F$ defined on subsets of Sobolev spaces can be locally approximated by a neural operator $F_\n$ in \autoref{eq:deeponet}. We emphasize that the nonlinear operator $\tilde{F}_\n$, as defined in \autoref{eq:nldeep} on the right hand side, is a specific form of $F_\n$, so the results apply to them as well.
\subsection{Approximation of Sobolev functions}

In this subsection, we extend the approximation of continuous functions to Sobolev spaces.
\begin{theorem}[Approximation of Sobolev functions] \label{th:function_apx_sob} 
	Let $\sigma$ be a sigmoidal activation function as defined in \autoref{def:sigm_act}. Let $s>\frac{n}{2}$ and $\Omega \subseteq \R^n$ be a bounded domain with piecewise $C^1$ boundary and satisfy the cone property (for a detailed definition of these properties see \cite{Ada75}). Let $\V=H^s(\Omega)$ be the  Sobolev space equipped with the Sobolev norm $\norms{\cdot}_{H^s}$. Moreover, let ${\tt K}$ be a bounded subset of $\V$. Then, for every $\ve >0$ and every ${\tt u} \in {\tt K}$, the assertions of \autoref{th:function apx} applies.
\end{theorem}
\begin{proof}
	The compact Sobolev embedding theorem (see \cite[Thm. 6.2]{Ada75}) states that the embedding from ${\tt V}=H^s(\Omega)$ into $C(\overline{\Omega})$ is compact. Therefore a bounded set ${\tt K}$ of $H^s(\Omega)$  is compact in $C(\overline{\Omega})$. Applying \autoref{th:function apx} proves the assertion.
\end{proof}

\subsection{Approximation of functionals in Sobolev spaces} 

In this subsection, we extend \autoref{th:functional_apx_cont} to Sobolev spaces.
\begin{theorem}[Approximation of functionals]\label{th:functional apx}
	Let $\sigma$ be a bounded sigmoidal activation function. Let $s>\frac{m}{2}$ and $\Omega_\X \subseteq \R^m$ be a bounded domain with piecewise $C^1$ boundary and satisfy the cone property (for a detailed definition of these properties see \cite{Ada75}). Moreover, let 
	\begin{equation*}
		F: \dom{F} \subseteq C(\overline{\Omega_\X}) \to \R
	\end{equation*}
	be a continuous functional on the space $C(\overline{\Omega_\X})$ and suppose that $\dom{F}$ is bounded and weakly closed with respect to the $H^s(\Omega_\X)$ norm. Then, for every $\ve >0$ and every $\x \in \dom{F}$ the assertions of \autoref{th:functional_apx_cont} hold.
\end{theorem}

\begin{proof}

	Since $\dom{F}$ is bounded with respect to the $H^s(\Omega_\X)$ norm
	every sequence $(\x_k)_{k \in \N}$ in $\dom{F}$ has a weakly convergent subsequence in $H^s(\Omega_\X)$, which we again denote by $(\x_k)_{k \in \N}$, and the limit is denoted by $\x$, which is in $\dom{F}$ because it is weakly closed. From the compact Sobolev embedding theorem (see \cite[Thm. 6.2]{Ada75}) we obtain that $(\x_k)_{k \in \N}$ is converging strongly (uniformly) to $\x$ in $C(\overline{\Omega_\X})$ with the $\norm{\cdot}_{L^\infty}$ norm. The continuity of $F$ gives the assertion by applying \autoref{th:functional_apx_cont}.
\end{proof}

\subsection{Approximation of operators in Sobolev spaces}\label{sec:sobolevproof}
Analogous to the approximation of functionals in Sobolev spaces, \autoref{th:operatorapxc} can be extended to the setting of Sobolev spaces as follows:

\begin{corollary}[Approximation of operators]\label{th:operatorapx}
	Let $\sigma$ be a sigmoid activation function and let $\Omega_\X$, $\Omega_\Y$ be bounded domains in $\R^m, \R^n$, respectively, where $\Omega_\X$ satisfies the cone property and has piecewise $C^1$ boundary. Moreover, let $s>\frac{m}{2}$ and let
	\begin{align*}
		F: \dom{F} \subseteq C(\overline{\Omega_\X}) &\to L^2(\Omega_\Y)
	\end{align*}
	be a continuous operator, where $\dom{F}$ is bounded and weakly closed with respect to the $H^s(\Omega_\X)$ norm. Then, for every $\ve > 0$ and every $\x \in \dom{F}$ the assertions of \autoref{th:operatorapxc} hold.
\end{corollary}
The proof is analogous to the proof of \autoref{th:function apx}.

In the following we apply these approximation results to regularization theory:

\subsection{Regularization in Sobolev spaces}
Returning to inverse problems, we apply Tikhonov regularization with a neural operator $F_\n$ to obtain a regularized solution $\x_\n^{\alpha\delta\eta}$. The result is conceptually similar to the linear case (see \autoref{th:linear}):
\begin{theorem}[$F$ non-linear]  \label{thm:sobolevnonlin}
	Let $s > \frac{m}{2}$ and let $F: \dom{F} \to L^2(\Omega_\Y)$ be a mapping, which satisfies the following properties:

\begin{itemize}
	   \item There exists ${\tt r}>0$ such that $\mathcal{B}_{\tt r}^{H^s}(\xdag) \subseteq \dom{F}$ with  ${\tt r} > 2\norms{\x^{(0)}-\xdag}_\X$ and $\dom{F}$ is bounded and weakly closed in $H^s(\Omega_\X)$.
		\item $F : \dom{F} \subseteq C(\overline{\Omega_\X}) \to L^2(\Omega_\Y)$ is continuous. Recall the notation stated after \autoref{eq:op} that $\X=C(\overline{\Omega_\X})$ determines the mapping properties of the operator $F$.
		\item $F : \dom{F} \subseteq H^s(\Omega_\X) \to L^2(\Omega_\Y)$ is weakly sequentially closed and Fr\'echet differentiable with Lipschitz continuous derivative on $\mathcal{B}^{H^s}_{\tt r}(\xdag)$ (see \autoref{eq:lipschitz}).
\end{itemize} 
Let $\x^{(0)}$ satisfy the following source condition
	\begin{equation} \label{eq:sourcenl}
		\x^{(0)}-\xdag = F'[\xdag]^* \omega \text{ with } L_1\norms{\omega}_\Y < 1,
	\end{equation}
	where $L_1$ denotes the Lipschitz constant of the Fr\'echet derivative of $F$ in $\mathcal{B}_{\tt r}(\xdag)$ (see \autoref{eq:lipschitz} and \cite{NeuSch90}).

	Then, for every $\ve > 0$, there exists some $\n(\ve) \in \N$ and a vector $\mathcal{T}_{\n(\ve)}$ of size $\n(\ve)$ such that the according operator $F_{\n(\ve)}$ from \autoref{eq:deeponet} satisfies
\begin{equation} \label{eq:nun}
		\rho_{\n(\ve)} := \sup_{\x \in \mathcal{B}_{\tt r}^{H^s}(\xdag)} \norms{F[\x]-F_{\n(\ve)}[\x]}_{L^2} \leq \ve\;.
\end{equation}
Let $\x_{\n}^{\alpha\delta\eta}$ be an approximate minimizer with accuracy $\eta >0$ of the approximate Tikhonov functional $\Tik_{\n}^{\alpha\delta}$ defined in \autoref{eq:discTik_s}. Then,
	with the choice 

	\begin{equation} \label{eq:optimal}
		\alpha \sim \delta, \eta \sim \alpha^2,  \text{ and } \ve \sim \delta\,,
	\end{equation} 
	we obtain the convergence rate 
	\begin{equation} \label{eq:rate}
		\norms{\x_\n^{\alpha\delta\eta}-\xdag}_{\X} = \mathcal{O}(\sqrt{\delta}) 
	\end{equation} 
	of the regularized solution.
\end{theorem}
\begin{proof}
\autoref{eq:nun} holds as an application of \autoref{th:operatorapx}. Then,  we can follow the technique in the proof of \cite[Theorem 2.3a]{NeuSch90} to obtain \autoref{eq:rate}.

\end{proof}
\begin{remark}
	For a linear operator $F$, $\rho_{\n(\ve)}$ in \autoref{eq:nun} is put in relation with \autoref{cor:rate}. The operator approximation error is therefore given by 
	\begin{equation}\label{eq:auslauf}
	\rho_{\n(\ve)} = \nu_\no \left( \norm{\xdag}_\X + {\tt r} \right) +  \no q(\no) r(\no) \leq \ve\,.
	\end{equation}
	Here, one sees that the number of coefficients $\n(\ve)$ of the operator $\tilde{F}_{\n(\ve)}$ (see \autoref{eq:nldeepm}) are related to the number of training pairs $\no$. Note that $F_{\n(\ve)}$ is a nonlinear approximation of a linear operator $F_{\no}^\#$, which can be identified with the general form of a neural operator $F_{\n}$ (see \autoref{eq:nldeepm}).
	In fact, the first part of the estimate $\nu_\no \left( \norm{\xdag}_\X + {\tt r} \right)$ is the estimate for the projection error $\norms{F-F_\no^\#}$ (in the operator norm) on the subspace spanned by the training images (see \autoref{eq:estimate}). In the situation of a linear operator $F$ we can explicitly determine the number of necessary training data from a prescribed accuracy $\ve$.\\
    In the case of a nonlinear operator $F$ (because of the first term on the right hand side in \autoref{eq:estimate}, which is only applicable for linear operators), we cannot show the same rate so far. However, in most applications, this seems to be a good estimate. We emphasize the different role of $\n$, which specifies how many coefficients $J,K,L$ are used to approximate the operator $F$ (see \autoref{eq:deeponet}).
\end{remark}

In the following, we further investigate \autoref{ex: darcyflow}:
\begin{example} \label{eg:aNN}
	We continue with \autoref{ex: darcyflow}, assuming that the source condition in \autoref{de:source_a} is satisfied. Instead of a classical finite element based approach, we consider Tikhonov regularization with a neural operator approximation. That is, we compute the approximate minimizer of the functional $\Tik_{\n}^{\alpha\delta}$, defined in \autoref{eq:discTik_s} with a neural operator as defined in \autoref{eq:deeponet}.

We consider $F$ on the restricted domain $\mathcal{B}_{\tt r}^{H^1}(\xdag) \subseteq \dom{F}:= \set{\x \in H^1(0,1): \x \geq \nu>0}$, which is always possible for ${\tt r}$ sufficiently small. Because of the compact embedding theorem, we find that $\mathcal{B}_{\tt r}^{H^1}(\xdag)$ is compact in $C(\overline{\Omega_\X})$, so that the assumptions in \autoref{thm:sobolevnonlin} holds.

Then, by applying \autoref{thm:sobolevnonlin}, according to \autoref{eq:nun}, for every $\ve$, there exists $\n:= \n(\ve) \in \N$, such that 
	\begin{equation*}
		\norms{F[\x]-F_{\n}[\x]}_{L^2} \leq \ve
	\end{equation*}
holds locally uniformly with respect to $\x$ in a neighborhood of $\xdag$, and with the choice of parameters (note that $\ve \to 0$)
	\begin{equation*}
		\alpha \sim \delta, \; \eta \sim \delta^2 \,, \ve \sim \delta,
	\end{equation*}
    we obtain the convergence rate
    \begin{equation}\label{eq:ex_a_nn}
        \norms{\x_\n^{\alpha\delta\eta}-\xdag}_{H^1} = \mathcal{O}(\sqrt{\delta})\;. 
    \end{equation} 
    This parameter choice relates the total number of coefficients of the neural operator $\n$ used to approximate $F$, the noise level and the regularization parameter. In contrast to the results in \cite{NeuSch90}, the number of coefficients $\n=\n(\ve)$ of the neural operator $F_\n$ is not explicitly given but determined \emph{implicitly} from the accuracy $\ve$.
\end{example}

For the $c$-problem from \autoref{ex:2point}, we need to generalize the concept of neural operators to $L^2$-functions, which is discussed in the following section.

\section{Neural operators in $L^2$} \label{se:generalization}
Instead of approximate minimization of the discretized Tikhonov functional defined in \autoref{eq:discTik_s}, we consider now finding an approximate minimizer of the functional
\begin{equation} \label{eq:discTikxi}
	\Tik_{\xi\n}^{\alpha\delta}[\x] := \norms{F_\n[\x]-\yd}_\Y^2 + \alpha\norms{\x-\x^{(0)}}^2_{\X}\,,
\end{equation} 
over the set $\dom{F} \cap \X_\xi$, where $\dom{F} = \dom{F_\n} \subseteq L^2(\Omega_\X)$ (the $\dom{F_\n}$ is independent of $\n$) and 
\begin{equation*}
	\X_\xi := \set{\x_\xi: \x \in L^2(\Omega_\X)}\,,
\end{equation*}
denotes the set of \emph{mollified functions} at a fixed smoothing level $\xi > 0$: To define the function $\x_\xi$, we extend $\x \in L^2(\Omega_\X)$ by zero on $\R^m\backslash\Omega_\X$ and define
\begin{equation}\label{def:mollifier}
	\bs \in \R^m \mapsto \x_\xi(\bs) := M_\xi\x:=(\phi_\xi*\x)(\bs) = \int_{\R^m} \phi_\xi(\bs - \bt)\x(\bt) \ \mathrm{d}\bt, 
\end{equation}
with the family of \emph{mollifiers}
$$\phi_\xi(\bs) := \frac{1}{\xi^m}\phi\left(\frac{\bs}{\xi}\right)\,,$$
which are defined by the standard mollifier,
\begin{equation*}
	\bs \mapsto \phi(\bs) := 
	\begin{cases}
		C \exp\left( \displaystyle\frac{1}{|\bs|^2 - 1} \right) & \text{if } |\bs| < 1 \\
		0 & \text{if } |\bs| \geq 1
	\end{cases} \in C_c^{\infty}(\R^m)\,,
\end{equation*}
where the constant $C>0$ is chosen such that $\int_{\R^m}\phi(\bs) \d \bs = 1.$

Due to the continuity of mollified functions, neural operators, as defined in \autoref{de:deeponet}, can be applied because point evaluations make sense for mollified functions. This means that every term in \autoref{eq:discTikxi} is well-defined.

The mollified functions have the following properties (see for instance \cite{AdaFou03,Eva10}): 
\begin{lemma} \label{le:molli} Let $\x \in L^2(\Omega_\X)$. Then 
	\begin{enumerate}
		\item $\x_\xi \in C^{\infty}(\R^m) \subseteq C(\overline{\Omega_\X})$.
		\item $\norms{\x_\xi}_{L^2(\Omega_\X)} \leq \norms{\x_\xi}_{L^2(\R^m)} \leq \norms{\x}_{L^2(\R^m)} = \norms{\x}_{L^2(\Omega_\X)}$.
		\item 
        $\norms{\x_\xi - \x}_{L^2(\Omega_\X)} \leq \norms{\x_\xi - \x}_{L^2(\R^m)}  \xrightarrow{\xi \to 0} 0$.
		\item The \emph{mollification operator}
			\begin{equation*}
				M_\xi: L^2(\Omega_\X) \to C(\overline{\Omega_\X})\,, \quad \x \mapsto \x_\xi\,,
			\end{equation*}
			where $C(\overline{\Omega_\X})$ is equipped with the $L^\infty(\Omega_\X)$ norm, is linear and bounded.
	\end{enumerate}
\end{lemma}
We show the well-definedness of the mollified Tikhonov-regularized solutions, that are the minimizers of the Tikhonov-functional defined in \autoref{eq:discTikxi}. For this purpose we prove several lemmata:
\begin{lemma}\label{le:l2 operator apx}
	Let $F: \mathcal{D}(F) \subseteq \X=L^2(\Omega_\X) \to L^2(\Omega_\Y)$ be weakly sequentially closed, continuous and Fr\'echet differentiable with Lipschitz continuous derivative in the open $L^2$ ball $\mathcal{B}^{L^2}_{\tt r}(\xdag)$ 
		(see \autoref{eq:lipschitz}). Then, there exists some $\xi^\dagger > 0$ such that for all
		\begin{equation} \label{eq:xidag}
			0 < \xi \leq \xi^\dagger\,,
		\end{equation}  
		\begin{equation} \label{eq:bound}
			\xdag_\xi \in \mathcal{B}^{L^2}_{\frac{\tt r}{2}}(\xdag)\;.
		\end{equation} 		
		Moreover, for all $\x \in \mathcal{B}_{\frac{\tt r}{2}}^{L^2}(\xdag)$,
		\begin{equation}\label{eq:help}
		\begin{aligned}
			\norm{\x_\xi-\xdag}_{L^2(\Omega_\X)} < \tt r\;.
		\end{aligned}
		\end{equation} 
		and
		\begin{equation} \label{eq:help1}
		\begin{aligned}
			\norm{F[\x_\xi]-F[\x]}_{L^2(\Omega_\Y)} &\leq L_0 \norm{\x_\xi-\x}_{L^2(\Omega_\x)}\,,
		\end{aligned}
		\end{equation}
		where $L_0$ is defined in \autoref{eq:lipschitz0}. 
\end{lemma}
\begin{proof}
	The first assertion \autoref{eq:bound} is an easy consequence of \autoref{le:molli}(iii). 
	Secondly, we prove that for every $\x \in \mathcal{B}_{\frac{\tt r}{2}}^{L^2}(\xdag)$, 
	$\x_\xi \in \mathcal{B}_{{\tt r}}^{L^2}(\xdag)$. Which in particular means that we can apply $F$ to $\x_\xi$, that it is Fr\'echet differentiable there and Lipschitz estimates can be applied in a neighborhood of $\x_\xi$. Because of \autoref{le:molli} (iii) applied to $\x_\xi-\xdag_\xi$ (which is possible because the convolution is linear) and \autoref{eq:bound}, it follows that
	\begin{equation*}
		\begin{aligned}
			\norm{\x_\xi-\xdag}_{L^2(\Omega_\X)} &\leq \norms{\x_\xi-\xdag_\xi}_{L^2(\Omega_\X)} + \norms{\xdag-\xdag_\xi}_{L^2(\Omega_\X)} \\
			&\leq \norms{\x-\xdag}_{L^2(\Omega_\X)} + \norms{\xdag-\xdag_\xi}_{L^2(\Omega_\X)} < \frac{\tt r}{2} + \frac{\tt r}{2} = {\tt r}\;.
		\end{aligned}
	\end{equation*} 
	Therefore, $\x_\xi \in \mathcal{B}_{\tt r}^{L^2}(\xdag)$ (showing \autoref{eq:help}) and for $\x \in \mathcal{B}_{\frac{\tt r}{2}}^{L^2}(\xdag) \subseteq \mathcal{B}_{\tt r}^{L^2}(\xdag)$, we can apply \autoref{eq:lipschitz0} and get our desired estimate \autoref{eq:help1}.
\end{proof}

\begin{lemma}\label{le:l2 operator apxa}
	Let $F: \mathcal{D}(F) \subseteq \X=L^2(\Omega_\X) \to L^2(\Omega_\Y)$ be weakly sequentially closed, continuous and Fr\'echet differentiable with Lipschitz continuous derivative in $\mathcal{B}^{L^2}_{\tt r}(\xdag)$ (see \autoref{eq:lipschitz}). Moreover, for every $\xi > 0$, $\dom{F}$ satisfies $M_\xi(\dom{F}) \subseteq \dom{F}$ and $M_\xi(\dom{F})$ is compact in $(C(\overline{\Omega_\X}),\norm{\cdot}_{L^\infty(\Omega_\X)})$.

Then, for every $\xi > 0$ satisfying \autoref{eq:xidag} and for every $\ve > 0$, 
	there exists $\n:=\n(\xi,\ve) \in \N$ such that 	
    \begin{equation} \label{eq:help3}
			\begin{aligned}
				\sup_{\set{\x_\xi:\norm{\x-\xdag}_{L^2(\Omega_\X)} < \frac{\tt r}{2}}} \norm{F_{\n}[\x_\xi]-F[\x_\xi]}_{L^2(\Omega_\Y)} &\leq \ve\;.\\
			\end{aligned}
		\end{equation}
\end{lemma}
\begin{proof}
Let $0< \xi \leq \xi^\dagger$ be arbitrary. For all $\x \in \mathcal{B}_{\frac{\tt r}{2}}^{L^2}(\xdag)$, we obtain from Cauchy-Schwarz inequality that 
\begin{equation} \label{eq:unter}
\begin{aligned}
	\abs{\x_\xi(\bs) - \xdag_\xi(\bs)} &= \abs{\int_{\R^m} \phi_\xi(\bs - \bt)\left(\x(\bt)- \xdag(\bt)\right)  \mathrm{d}\bt} \leq \norm{\phi_\xi}_{L^2(\R^m)}\norm{\x - \xdag}_{L^2(\R^m)} \\ &\leq \xi^{-\frac{m}{2}} \norm{\phi}_{L^2(\R^m)} \frac{\tt r}{2} \text{ for all } \bs \in \Omega_\X\;.
\end{aligned}
\end{equation}
The inequality $\norm{\phi_\xi}_{L^2(\R^m)} \leq \xi^{-\frac{m}{2}} \norm{\phi}_{L^2(\R^m)}$ holds, since
\begin{align*}
\norm{\phi_\xi}_{L^2(\R^m)}^2 = \int_{\R^m} \xi^{-2m}\abs{\phi\left(\frac{\bs}{\xi}\right)}^2 \mathrm{d}\bs = \int_{\R^m}  \xi^{-2m}\abs{\phi(\bs)}^2 \xi^m \mathrm{d}\bs =  \xi^{-m}\int_{\R^m}  \abs{\phi(\bs)}^2 \mathrm{d}\bs \leq \xi^{-m}\norm{\phi}^2_{L^2(\R^m)}\;.
\end{align*}
Since $\x$ and $\xdag$ are extend by zero outside of $\Omega_\X$ it follows that
 $$\norm{\x - \xdag}_{L^2(\R^m)} = \norm{\x - \xdag}_{L^2(\Omega_\X)} < \frac{\tt r}{2}.$$
Note that the standard mollifier $\phi$ and all of its derivative are $L^2(\R^m)$-functions, since $\phi$ is a smooth and compactly supported by definition.
It follows from \autoref{eq:unter} that
\begin{align*}
	\norms{\x_\xi - \xdag_\xi}_{L^\infty(\Omega_\X)} \leq \xi^{-\frac{m}{2}} \norm{\phi}_{L^2(\R^m)} \frac{\tt r}{2} := r(\xi)\,,
\end{align*}
or in other words
 
$\x_\xi \in \mathcal{B}_{r(\xi)}^{L^\infty}(\xdag_\xi) \subseteq \X_\xi$. In summary, we have shown that
\begin{equation} \label{eq:helpthml2}
M_\xi\left(\mathcal{B}_{\frac{\tt r}{2}}^{L^2}(\xdag)\right) \subseteq \mathcal{B}_{r(\xi)}^{L^\infty}(\xdag_\xi)\;.
\end{equation}
Now, we want to prove the compactness in $C(\overline{\Omega_\X})$ of $\mathcal{B}_{r(\xi)}^{L^\infty}(\xdag_\xi)$. Analogously, as for \autoref{eq:unter}, we obtain that 
\begin{align*}
\norm{\x_\xi}_{L^\infty(\Omega_\X)} &\leq \xi^{-\frac{m}{2}} \norm{\phi}_{L^2(\R^m)} \norm{\x}_{L^2(\Omega_\X)} \leq \xi^{-\frac{m}{2}} \norm{\phi}_{L^2(\R^m)} \left(\norm{\x - \xdag}_{L^2(\Omega_\X)} + \norm{\xdag}_{L^2(\Omega_\X)}\right) \\
&\leq \xi^{-\frac{m}{2}} \norm{\phi}_{L^2(\R^m)} \left(\frac{\tt r}{2} + \norm{\xdag}_{L^2(\Omega_\X)}\right) \text{ for all } \x_\xi \in \mathcal{B}_{r(\xi)}^{L^\infty}(\xdag_\xi)
\end{align*}
and analogously,
\begin{align*}
\norm{\nabla\x_\xi}_{L^\infty(\Omega_\X)} \leq \xi^{-\frac{m}{2}-1} \norm{\nabla\phi}_{L^2(\R^m)} \left(\frac{\tt r}{2} + \norm{\xdag}_{L^2(\Omega_\X)}\right).
\end{align*}
Since for fixed $0 < \xi \leq \xi^\dagger$, $\mathcal{B}_{r(\xi)}^{L^\infty}(\xdag_\xi)$ is uniformly bounded and equicontinuous with respect to the $(C(\overline{\Omega_\X}),\norm{\cdot}_{L^\infty(\Omega_\X)})$-topology, it follows from the theorem of Arzel\`a-Ascoli \cite{Rud76} that $\mathcal{B}_{r(\xi)}^{\infty}(\xdag_\xi)$ is compact in $C(\overline{\Omega_\X})$. Note that $\mathcal{B}_{r(\xi)}^{L^\infty}(\xdag_\xi)$ is getting larger as $\xi \to 0$.

In order to show that $F$ is continuous in $M_{\xi}(\dom{F})$ with respect to $\norm{\cdot}_{L^\infty(\Omega_\X)}$ on $C(\overline{\Omega_\X})$, we prove that, if $\x_{\xi,n} \to \x_\xi$ in $M_\xi(\dom{F})$ with respect to $\norm{\cdot}_{L^\infty(\Omega_\X)}$, then it follows that $F[\x_{\xi,n}] \to F[\x_\xi]$ in $L^2(\Omega_\Y)$: Since $\Omega_\X$ is bounded, it follows that there exists a constant $C>0$ such that

$$\norm{\x_{\xi,n} - \x_\xi}_{L^2(\Omega_\X)} \leq C \norm{\x_{\xi,n} - \x_\xi}_{L^\infty(\Omega_\X)}\;.$$

Since $M_\xi(\dom{F}) \subseteq \dom{F}$ and $\x_{\xi,n} \to \x_\xi$ in $C(\overline{\Omega_\X})$ by assumption, it follows that $\x_{\xi,n} \to \x_\xi$ in $\dom{F}$ with respect to $\norm{\cdot}_{L^2(\Omega_\X)}$ and  $\x_\xi \in \dom{F} \subseteq L^2(\Omega_\X)$. Due to the assumed continuity of $F$ in $\dom{F}$ with respect to $L^2(\Omega_\X)$, it follows that $F[\x_{\xi,n}] \to F[\x_\xi]$ in $L^2(\Omega_\Y)$ and we have proven the assertion.

With the compactness of $\mathcal{B}_{r(\xi)}^{L^\infty}(\xdag_\xi)$ in $C(\overline{\Omega_\X})$, and the continuity of $F$ in $M_\xi(\dom{F})$ with respect to $\norm{\cdot}_{L^\infty(\Omega_\X)}$, we can apply \autoref{th:operatorapxc} to obtain: For every $0 < \xi \leq \xi^\dagger$ and for every $\ve > 0$, there exists a coefficients tuple $\mathcal{T}_{\n(\xi,\ve)}$ (as defined in \autoref{eq:D}) and $\n := \n(\xi,\ve) = \abs{\mathcal{T}_{\n(\xi,\ve)}} \in \N$, such that the operator $F_{\n(\xi,\ve)} := F_{\n}$ from \autoref{eq:deeponet} satisfies
          \begin{equation} \label{eq:helthml2_2}
             \norms{F[\x_\xi] - F_\n[\x_\xi]}_{L^2(\Omega_\Y)} < \ve \text{ for all } \x_\xi \in \mathcal{B}_{r(\xi)}^{L^\infty}(\xdag_\xi).
        \end{equation}
        From \autoref{eq:helpthml2}, we have that \autoref{eq:helthml2_2} holds for all $\x_\xi \in M_\xi\left(\mathcal{B}_{\frac{\tt r}{2}}^{L^2}(\xdag)\right).$ Taking the supremum on the left hand side gives the assertion.
\end{proof}
\begin{remark} We emphasize that \autoref{eq:help3} holds for fixed $\xi$. A uniform estimate with respect to $\ve$ does not hold.
\end{remark}
In the following we derive approximation properties of a neural operator $F_\n$ as defined in \autoref{eq:deeponet} on the space $\X_\xi$. This setting distinguishes the result from previous results.

\begin{theorem}[$F_\n$ from \autoref{eq:deeponet} in $L^2$]\label{th:l2 operator apx}
	Let $F: \mathcal{D}(F) \subseteq \X=L^2(\Omega_\X) \to L^2(\Omega_\Y)$ be weakly sequentially closed, continuous and Fr\'echet differentiable with Lipschitz continuous derivative in $\mathcal{B}^{L^2}_{\tt r}(\xdag)$ (see \autoref{eq:lipschitz}). Moreover, for every $\xi > 0$, $\dom{F}$ satisfy $M_\xi(\dom{F}) \subseteq \dom{F}$ and $M_\xi(\dom{F})$ is compact in $C(\overline{\Omega_\X})$. Then, for every $\xi > 0$ satisfying \autoref{eq:xidag} and every $\ve > 0$, there exists a number $n(\xi,\ve)\in \N$, a coefficient vector $\mathcal{T}_{\n(\xi,\ve)}$ with associated operator $F_{\n(\xi,\ve)}$ from \autoref{eq:deeponet}, such that	
	\begin{equation}\label{eq:deeponeterror_l2}
		\norms{F[\x] - F_{\n(\xi,\ve)}[\x_{\xi}]}_{L^2(\Omega_\Y)} \leq \ve + L_0 \norm{\x - \x_{\xi}}_{L^2(\Omega_\X)} \text{ for all } \x \in \mathcal{B}_{\frac{\tt r}{2}}^{L^2}(\xdag)\;,
	\end{equation}
where $L_0$ denotes the Lipschitz constant of $F$ in $\mathcal{B}_{\tt r}(\xdag)$ (see \autoref{eq:lipschitz0}).
\end{theorem}
\begin{proof}
From \autoref{le:l2 operator apx}, we get the uniform estimate 
\begin{equation} \label{eq:hi1}
	\norm{F[\x]-F[\x_\xi]}_{L^2(\Omega_\Y)} \leq L_0 \norm{\x - \x_\xi}_{L^2(\Omega_\X)} \text{ for all } \x \in \mathcal{B}_{\frac{\tt r}{2}}^{L^2}(\xdag) \text{ and } 0 < \xi < \xi^\dagger\;.
\end{equation}
From \autoref{le:l2 operator apxa}, it follows that for every $0 < \xi < \xi^\dagger$ and every $0 < \ve$, there exists $\n(\xi,\ve) \in \N$ and a corresponding operator $F_{\n(\xi,\ve)}$ from \autoref{eq:deeponet} such that
\begin{equation} \label{eq:hi2}
	\norm{F[\x_{\xi}] - F_{\n(\xi,\ve)}[\x_{\xi}]}_{L^2(\Omega_\Y)} \leq \ve  \text{ for all } \x \in \mathcal{B}_{\frac{\tt r}{2}}^{L^2}(\xdag)\;.
\end{equation}
Combining \autoref{eq:hi1} and \autoref{eq:hi2}, we find
\begin{equation*}
	\begin{aligned}
		\norms{F[\x] - F_{\n(\xi,\ve)}[\x_{\xi}]}_{L^2(\Omega_\Y)} \leq \ve + L_0\norm{\x - \x_{\xi}}_{L^2(\Omega_\X)} \text{ for all } \x \in \mathcal{B}_{\frac{\tt r}{2}}^{L^2}(\xdag)\;.
	\end{aligned}
	\end{equation*}
\end{proof}
With the proven approximation properties in \autoref{eq:deeponeterror_l2_dup}, we can prove a convergence rates result of Tikhonov-regularized solutions minimizing the functional $\Tik_{\xi\n}^{\alpha\delta}$, defined in \autoref{eq:discTikxi}.

\begin{theorem} \label{thm:lebesguenonlin}
		Let $F: \dom{F} \subseteq \X= L^2(\Omega_\X) \to L^2(\Omega_\Y)$ be a mapping, which satisfies the following properties: For every $0 < \xi \leq \xi^\dagger$
\begin{itemize}
		\item $M_\xi(\dom{F}) \subseteq \dom{F}$.
		\item $M_\xi(\dom{F})$ is compact in $C(\overline{\Omega_\X})$.
		\item  There exists ${\tt r}>0$ such that $\mathcal{B}_{\tt r}^{L^2}(\xdag) \subseteq \dom{F}$ with  ${\tt r} > 2\norms{\x^{(0)}-\xdag}_{L^2(\Omega_\X)}$ and $\dom{F}$ is bounded and weakly closed in $L^2(\Omega_\X)$.
		\item $F : \dom{F} \subseteq L^2(\Omega_\X) \to L^2(\Omega_\Y)$ is continuous.
		\item $F : \dom{F} \subseteq L^2(\Omega_\X) \to L^2(\Omega_\Y)$ is weakly sequentially closed and Fr\'echet differentiable with Lipschitz continuous derivative on $\mathcal{B}^{L^2}_{\tt r}(\xdag)$ (see \autoref{eq:lipschitz}).
\end{itemize}
Moreover let $\x^{(0)}$ satisfy the following source condition
	\begin{equation} \label{eq:sourcenl2}
		\x^{(0)}-\xdag = F'[\xdag]^* \omega \text{ with } L_1\norms{\omega}_\Y < 1,
	\end{equation}
	where $L_1$ denotes the Lipschitz constant of the Fr\'echet derivative of $F$ in $\mathcal{B}_{\tt r}(\xdag)$ (see \autoref{eq:lipschitz} and \cite{NeuSch90}).
Then, for every $0 < \xi < \xi^\dagger$ (satisfying \autoref{eq:xidag}) and every $\ve > 0$, there exists a coefficient vector $\mathcal{T}_{\n(\xi,\ve)}$ with associated operator $F_{\n(\xi,\ve)}$ from \autoref{eq:deeponet}, such that	
	\begin{equation}\label{eq:deeponeterror_l2_dup}
		\norms{F[\xdag] - F_{\n(\xi,\ve)}[\x_{\xi}^\dagger]}_{L^2(\Omega_\Y)} \leq \ve + \rho(\xi)\,, 
	\end{equation} 
	where
	\begin{equation} \label{eq:rhoxi}
			\rho(\xi) := L_0 \norms{\xdag - \xdag_{\xi}}_{L^2(\Omega_\X)},
	\end{equation}
	where $L_0$ denotes the Lipschitz constant of $F$ in $\mathcal{B}_{\tt r}(\xdag)$ (see \autoref{eq:lipschitz0}).
Moreover, for $\n := \n(\xi,\ve)$, let $\x_{\xi\n}^{\alpha\delta\eta}$ be an approximate minimizer of \autoref{eq:discTikxi} over the set $\X_\xi$ with accuracy $\eta > 0$. Then,
	with the choice 
	\begin{equation} \label{eq:choicel2}
		\alpha \sim \delta, \; \eta \sim \delta^2 \,, \ve \sim \delta \,, \rho(\xi) \leq \delta,
	\end{equation}
	we obtain the convergence rate 
	\begin{equation} \label{eq:ratel2}
		\norms{\x_\n^{\alpha\delta\eta}-\xdag}_{\X} = \mathcal{O}(\sqrt{\delta}) 
	\end{equation} 
	of the regularized solution.
\end{theorem}

\begin{proof} 	
	First we note that the neural operator $F_\n$ with $\n = \n(\xi,\ve)$ is well-defined for all $\x_\xi \in M_{\xi}(\dom{F})$ because every $\x_\xi$ is continuous and therefore point evaluation in $F_\n$ makes sense.

\autoref{eq:deeponeterror_l2_dup} directly follows from \autoref{eq:deeponeterror_l2}, since $\x_\xi^\dagger \in \mathcal{B}_{{\tt r}/2}^{L^2}(\xdag)$ (see \autoref{eq:bound}).

Now, to show $\autoref{eq:ratel2}$, let $\x_{\xi\n}^{\alpha\delta\eta}$ be an approximate minimizer of $\Tik_{\xi\n}^{\alpha\delta}$ with accuracy $\eta$. Then, by the definition of an approximate minimizer, we get
        \begin{equation*}
        	\begin{aligned}
        		\Tik_{\xi\n}^{\alpha\delta}[\x_{\xi\n}^{\alpha\delta\eta}] &= \norms{F_\n[\x_{\xi\n}^{\alpha\delta\eta}] - \y^\delta}^2_{L^2(\Omega_\Y)} + \alpha\norms{\x_{\xi\n}^{\alpha\delta\eta} - \x^{(0)}}^2_{L^2(\Omega_\X)}\\
        		&\leq \norms{F_\n[\xdag_\xi] - \y^\delta}^2_{L^2(\Omega_\Y)} + \alpha\norms{\xdag_\xi - \x^{(0)}}^2_{L^2(\Omega_\X)} + \eta\;.
        	\end{aligned}
        \end{equation*}
        Let $\rho(\xi)$ be as defined in \autoref{eq:rhoxi}. It follows that
        \begin{align} \label{eq:c1}
        \norms{\xdag_\xi - \x^{(0)}}^2_{L^2(\Omega_\X)} \leq \left(\norms{\xdag_\xi - \xdag}_{L^2(\Omega_\X)} + \norms{\xdag - \x^{(0)}}_{L^2(\Omega_\X)}\right)^2 \leq \left(L_0^{-1}\rho(\xi) + \norms{\xdag - \x^{(0)}}_{L^2(\Omega_\X)} \right)^2\;.% + \mathcal{O}(\xi^4).
    \end{align}
       Moreover, \autoref{eq:deeponeterror_l2_dup} gives
\begin{align} \label{eq:c2}
		 \norms{F_\n[\xdag_\xi] - \y^\delta}^2_{L^2(\Omega_\Y)} \leq 
            \left(\norms{F_\n[\xdag_\xi] - F[\xdag]}_{L^2(\Omega_\Y)} +\norms{F[\xdag] - \y^\delta}_{L^2(\Omega_\Y)}\right)^2
            \leq (\ve + \rho(\xi) + \delta)^2\;.
\end{align} 

Now, from \autoref{eq:c1} and \autoref{eq:c2}, we can follow the strategy of \cite[Theorem 2.3a]{NeuSch90} to get with the parameter choice \autoref{eq:choicel2}
    \begin{equation} \label{eq:xi}
        	\norms{\x_{\xi\n}^{\alpha\delta\eta} - \xdag}_{L^2(\Omega_\X)} = \mathcal{O}(\sqrt{\delta})\;.
    \end{equation}
\end{proof}

\begin{remark}
	In \cite{NeuSch90}, we assumed that $\norm{F[\x]-F_\n[\x_\xi]}_\Y \leq \rho_\n$ holds locally uniformly in a neighborhood of $\xdag$ in order to deduce an estimate for $\norms{F_\n[\xdag_\xi] - \y^\delta}_{L^2(\Omega_\Y)}$ that depends on $\rho_\n$ (similar to \autoref{eq:c2}). However, it is in fact sufficient to require that \autoref{eq:deeponeterror_l2} holds for $\x = \xdag$, which is the weaker condition \autoref{eq:deeponeterror_l2_dup}.
\end{remark}

Finally, we consider the error estimate for the c-example using neural operator approximation. 
\begin{example}[c-example with neural operator approximation] \label{eg:cNN}
We continue with \autoref{ex:2point}. Let $\xdag > \gamma > 0 \in H^2(0,1)$ and assume that $\x^{(0)} \in \X=L^2(\Omega_\X)$ satisfies the source condition in \autoref{de:source_c}. Instead of a classical finite element based approach, we consider Tikhonov regularization with a neural operator approximation.  That is, for all $0 < \xi \leq \xi^\dagger$ and $\ve > 0$, we compute the approximate minimizer $\x_{\xi\n}^{\alpha\delta\eta}$ of \autoref{eq:discTikxi} over the set $\X_\xi$, where we chose $\n := \n(\xi,\ve)$ accordingly.

The general assumptions in \autoref{sec:assump} are satisfied for $F$ and $\dom{F}$ (see \cite{NeuSch90}, Example 3.1), with the exception that $\mathcal{B}_{\tt r}^{L^2(\Omega)} \subseteq \dom{F}$. As it was shown in \cite{NeuSch90} this assumption can be circumvented for this particular example. To apply \autoref{thm:lebesguenonlin}, we need to additionally check that basic assumptions of  \autoref{thm:lebesguenonlin} that 
\begin{enumerate}
	\item $M_\xi(\dom{F}) \subseteq \dom{F} = \set{\x \in L^2(0,1): \x \geq 0 \text{ a.e.} }$ and
	\item $M_\xi(\dom{F})$ is compact in $C(\overline{\Omega_\X})$.
\end{enumerate} 
We prove this: 
\begin{enumerate}
\item $\dom{F}$ is a positive cone in $L^2(\Omega_\X)$, and convolution with positive kernel preserves positivity and the belonging to $L^2(\Omega_\X)$.
\item The operator $F$ defined in \autoref{eq:invillpos} can be extended to the domain
$$\mathcal{D}(F):= \{\x \in L^2(0,1) : \norm{\x - \hat{\x}}_{L^2} \leq \hat{\ve}, \hat{\x} \geq 0 \text{ a.e.}, \hat{\x} \in L^2(0,1)\},$$
for some fixed $ \hat{\epsilon} > 0$, while preserving the same properties as $F$ in \autoref{ex:2point} (see \cite{NeuSch90}). In the following we show that for fixed $0 < \xi \leq \xi^\dagger$, $M_\xi(\mathcal{D}(F))$ is uniformly bounded and equicontinuous in $C([0,1])$, so that we can apply the theorem of Arzel\`a-Ascoli \cite{Rud76} and obtain that $M_\xi(\mathcal{D}(F))$ is compact in $C([0,1])$. 
Analogously to \autoref{eq:unter}, we see that for fixed $\xi$ we have
$$\norm{\x_\xi}_{L^\infty(0,1)} \leq \xi^{\frac{-1}{2}}\norm{\phi}_{L^2(\R)}\norm{\x}_{L^2(0,1)} \leq \xi^{\frac{-1}{2}}\norm{\phi}_{L^2(\R)} \left( \norm{\hat{\x}}_{L^2(0,1)}  + \hat{\ve}\right)\;,$$
which shows the uniformly bounded of $M_\xi(\dom{F})$ in $C([0,1])$. Moreover, the derivative of $\x_\xi$ satisfies,
$$\norm{\x_\xi'}_{L^\infty(0,1)} \leq \xi^{\frac{-3}{2}}\norm{\phi'}_{L^2(\R)} \left( \norm{\hat{\x}}_{L^2(0,1)}  + \hat{\ve}\right)\;.$$
The derivative is uniformly bounded and therefore
$M_\xi(\mathcal{D}(F))$ is equicontinuous in $C([0,1])$.
\end{enumerate}

Now, we can apply \autoref{thm:lebesguenonlin}. Since $\xdag \in H^2(0,1)$, we have $\rho(\xi) = \mathcal{O}(\xi^2)$, where $\rho(\xi)$ is as defined in \autoref{eq:rhoxi}. With the choice of parameters
	\begin{equation*}
		\alpha \sim \delta, \; \eta \sim \delta^2 \,,  \ve \sim \delta \,, \xi \leq \sqrt{\delta}.
	\end{equation*}
    we obtain the convergence rate
    \begin{equation}
        \norms{\x_{\xi\n}^{\alpha\delta\eta} - \xdag}_{L^2(\Omega_\X)} = \mathcal{O}(\sqrt{\delta})\;
    \end{equation} 
\end{example}

\begin{remark}
	A conceptual difference between the quantitative results in \cite{NeuSch90} and those presented here is that the estimate \autoref{eq:deeponeterror_l2} depends on two parameters: the specified accuracy $\ve$ and the amount of mollification $\xi$. For finite element approximations in \cite{NeuSch90}, the estimate for $\norm{F[\x]-F_{\tt n}[\x]}_\Y$ is of order $\n^{-p}$ (depending on the smoothness of $\x$) for some given $p>0$. This means that we do not have to find the optimal relation between $\n(\xi,\ve)$; instead, it is explicitly given as a result of Cea's lemma \cite{Cia78} and the Aubin-Nitzsche trick \cite{Aub67,Nit68}.
\end{remark}
\begin{remark}
	Above, we have investigated how to apply neural operators to functions in $L^p$-spaces, which do not allow for pointwise evaluation everywhere. We have analyzed how we can apply regularization on ``presmoothed spaces'' $\X_\xi$. The technical difficulty arises from the fact that the numbers of neurons $\n$ depends on the amount of pre-smoothing $\xi$, which complicates the analysis considerably. 
\end{remark}

\section{Numerical experiments} \label{sec:numerics}
To verify our theoretical results, we compare solving the inverse problems in \autoref{ex: darcyflow} and \autoref{ex:2point} using Tikhonov regularization with the surrogate neural operators $F_{\tt n}$, as defined in \autoref{eq:deeponet}, and $F_\no^\#$, as defined in \autoref{eq:lin}, classical finite element surrogate operators.

\subsection*{Experimental setting}
To generate the training data $ \mathcal{S}_{F'[\hat{\x}^{(0)}]}^{(0)}$ as defined in \autoref{eq:S_G}, we use a Fourier sine basis for $\x^{(\ell)}$:
$$\x^{(\ell)}(s) = 1 + \sum_{k=1}^{\tt N} c_k\sin(2\pi k s) \quad \text{for} \quad s \in [0,1],$$
where ${\tt N}$ is the number of the training data. The coefficients $c_k$ are sampled from a normal distribution:
$$c_k \sim \mathcal{N}\left(0,\frac{1}{k^2}\right).$$
To ensure the positivity of $\x^{(\ell)}$, the coefficients are rescaled if necessary so that
$$\sum_{k=1}^{\tt N} \abs{c_k} < 1.$$

We then apply the Gram–Schmidt procedure to obtain an orthonormal set $\{\ulx^{(\ell)} : \ell = 1,\cdots, \no\}$ as in \autoref{eq:sample}. For each $\ulx^{(\ell)}$, the corresponding boundary value problem is solved numerically using the finite element method (FEM): \autoref{eq:darcyBVP} for \autoref{ex: darcyflow} and \autoref{eq:invillpos} for \autoref{ex:2point}. This yields the corresponding output $$\uly^{(\ell)} = F[\ulx^{(\ell)}].$$

The linear surrogate operator $F_\no^\#$, as defined in \autoref{eq:lin}, can be computed explicitly and does not require training.

The neural operator $F_\n$, as defined in \autoref{eq:deeponet} with sigmoid activation function $\sigma(s) = \frac{1}{1+e^{-s}}$, is trained on $\no=20$ synthetically generated samples $\{(\ulx^{(\ell)},\uly^{(\ell)})\}_{\ell=1}^{20}$ (i.e. ${\tt N} = 20$), with $J = K = L = 200$, which makes $\n \approx 10^7$. Training is performed for 1000 epochs, where each epoch corresponds to one full pass through the dataset. In principle, one can train $F_\n$ directly on the non-orthonormalized data pairs $(\x^{(\ell)},\y^{(\ell)})$. However, since the construction of $F_\no^\#$ relies on the orthonormal basis representation, and the data are synthetically generated by simultaneously solving the PDE, we perform the orthonormalization upfront so that the same set of PDE solves can be reused for both models.

After training the surrogate operator $F_\n$, we minimize the Tikhonov functional in \autoref{eq:discTik_s} using the AdamW optimizer, stopping either after 1000 iterations or when the change in the objective between consecutive steps falls below $10^{-6}$. For \autoref{ex: darcyflow}, we use a learning rate of $5 \times 10^{-2}$ and a weight decay of $10^{-3}$. For \autoref{ex:2point}, we use a weight decay of $10^{-4}$, a learning rate of $1$ for intial values $x^{(0)} \equiv 0$ and a learning rate of $10^{-2}$ for for intial values $x^{(0)} \equiv 1$. Moreover, for \autoref{ex:2point}, we minimize over the mollified class $\X_\xi$ as in \autoref{def:mollifier} with a mollification parameter $\xi = 5 \times 10^{-2}$.

For comparison, the inverse problem is also solved using FEM. The parameter is approximated in a finite element space consisting of continuous, piecewise linear functions defined on a uniform partition of $[0,1]$, with mesh size $h=1/\n$ with $\n=20,100$, corresponding to 20 and 100 elements, respectively. We compare the mean squared error $\norms{\x_\n - \x}_{L^2}^2$ between the reconstructed solution $\x_\n$ and the analytical solution $\x$ given in \autoref{eq:analytic} for \autoref{ex: darcyflow} and \autoref{eq:invillpos} for \autoref{ex:2point}.

In addition, the Tikhonov functional is minimized using both the noise-free input $\y$ and perturbed inputs $\y^\delta$ with $\delta = 0.03, 0.15$, in order to study the behavior of the algorithm with respect to noise.

The regularization parameter $\alpha$ is chosen as $\alpha = 5 \times 10^{-2}$ for noise-free data, and $\alpha=\delta$ for noisy data with noise level $\delta = 0.03, 0.15$.

\subsection*{Results and discussion}
\begin{enumerate}
    \item \textbf{For \autoref{ex: darcyflow}:} We solve the inverse problem using Tikhonov regularization, employing both a trained neural operator as a surrogate model and a finite element method. We consider three different initial guesses: $\x^{(0)} \equiv 0$ (\autoref{fig:result_x0=0}), $\x^{(0)} \equiv 1$ (\autoref{fig:result_x0=1}), and $\x^{(0)}(s) = s$ (\autoref{fig:result_x0=s}). The results are then validated by comparing them with the analytic solution of \autoref{eq:darcyBVP} for ${\tt f} \equiv 2$:
\begin{equation} \label{eq:analytic}
    \begin{cases}
        \x(s) = s+1\\
        \y(s) = \frac{2\ln(s+1)-s\ln(4)}{\ln(2)},
    \end{cases}
\end{equation}
which serves as the ground truth.

\begin{figure}[H]
    \centering
	\includegraphics[width=1\linewidth]{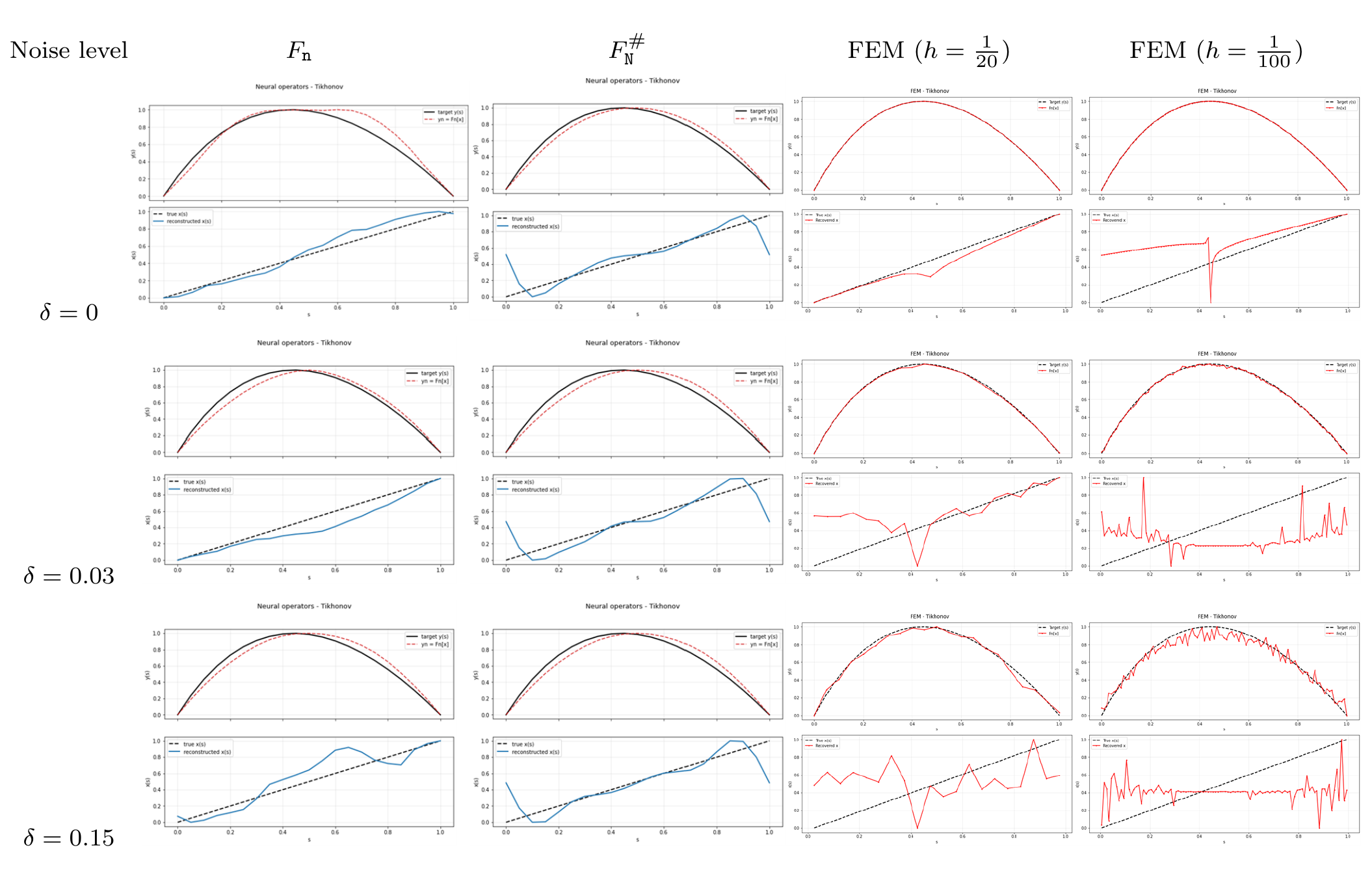}
    \caption{Results of solving the Tikhonov regularization problem with initial guess $\x^{(0)} \equiv 0$ using a neural surrogate operator $F_\n$ (first image), approximation operator $F_\no^\#$ (second image), and the finite element method with 20 elements (third image) and 100 elements (fourth image); without noise (first row), with noise $\delta = 0.03$ (second row) and $\delta=0.15$ (third row).}
    \label{fig:result_x0=0}
    
\end{figure}

\begin{figure}[H]
    \centering
	\includegraphics[width=1\linewidth]{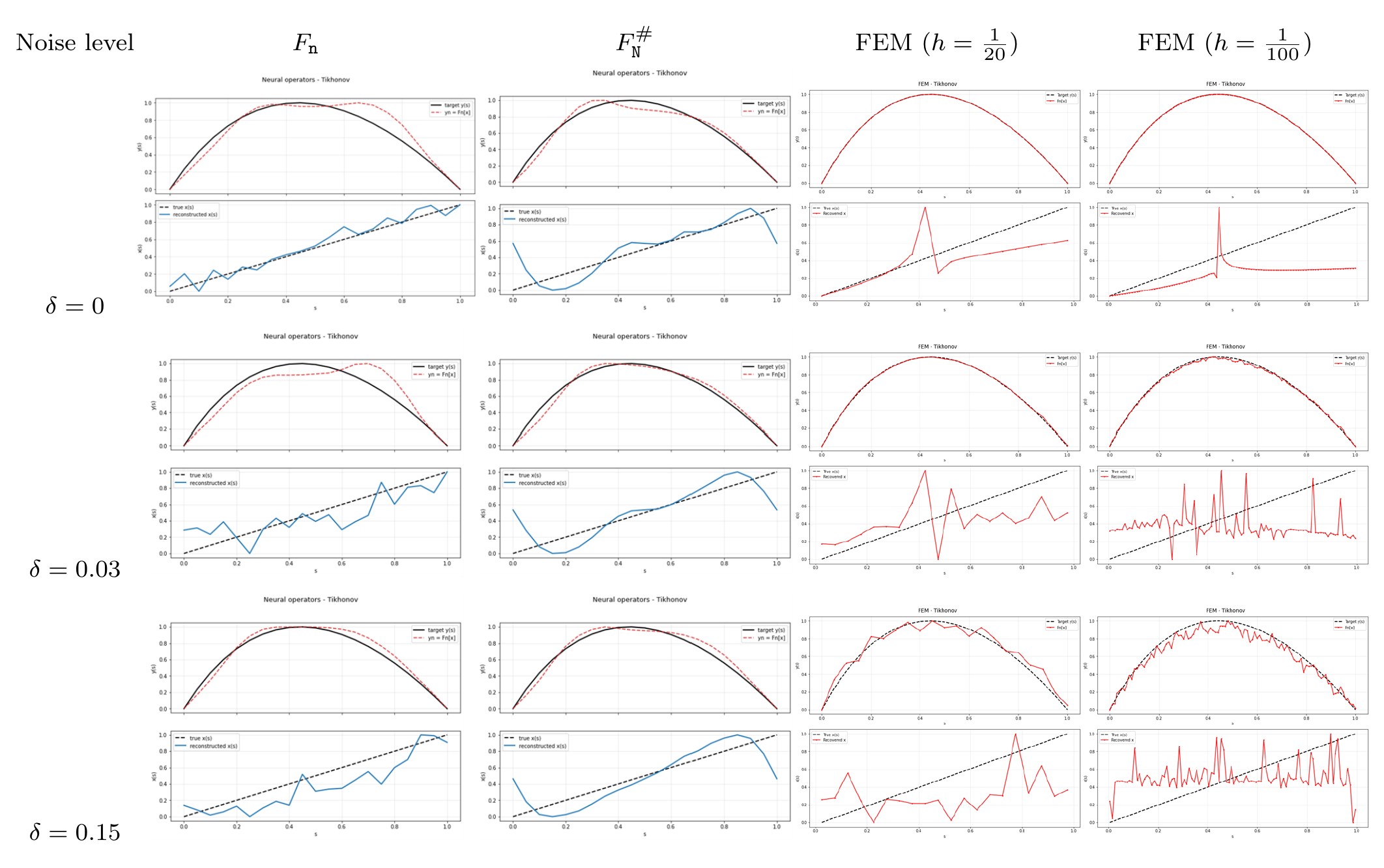}
    \caption{Results of solving the Tikhonov regularization problem with initial guess $\x^{(0)} \equiv 1$ using a neural surrogate operator $F_\n$ (first image), approximation operator $F_\no^\#$ (second image), and the finite element method with 20 elements (third image) and 100 elements (fourth image); without noise (first row), with noise $\delta = 0.03$ (second row) and $\delta = 0.15$ (third row).}
    \label{fig:result_x0=1}
\end{figure}

\begin{figure}[H]
    \centering
\includegraphics[width=1\linewidth]{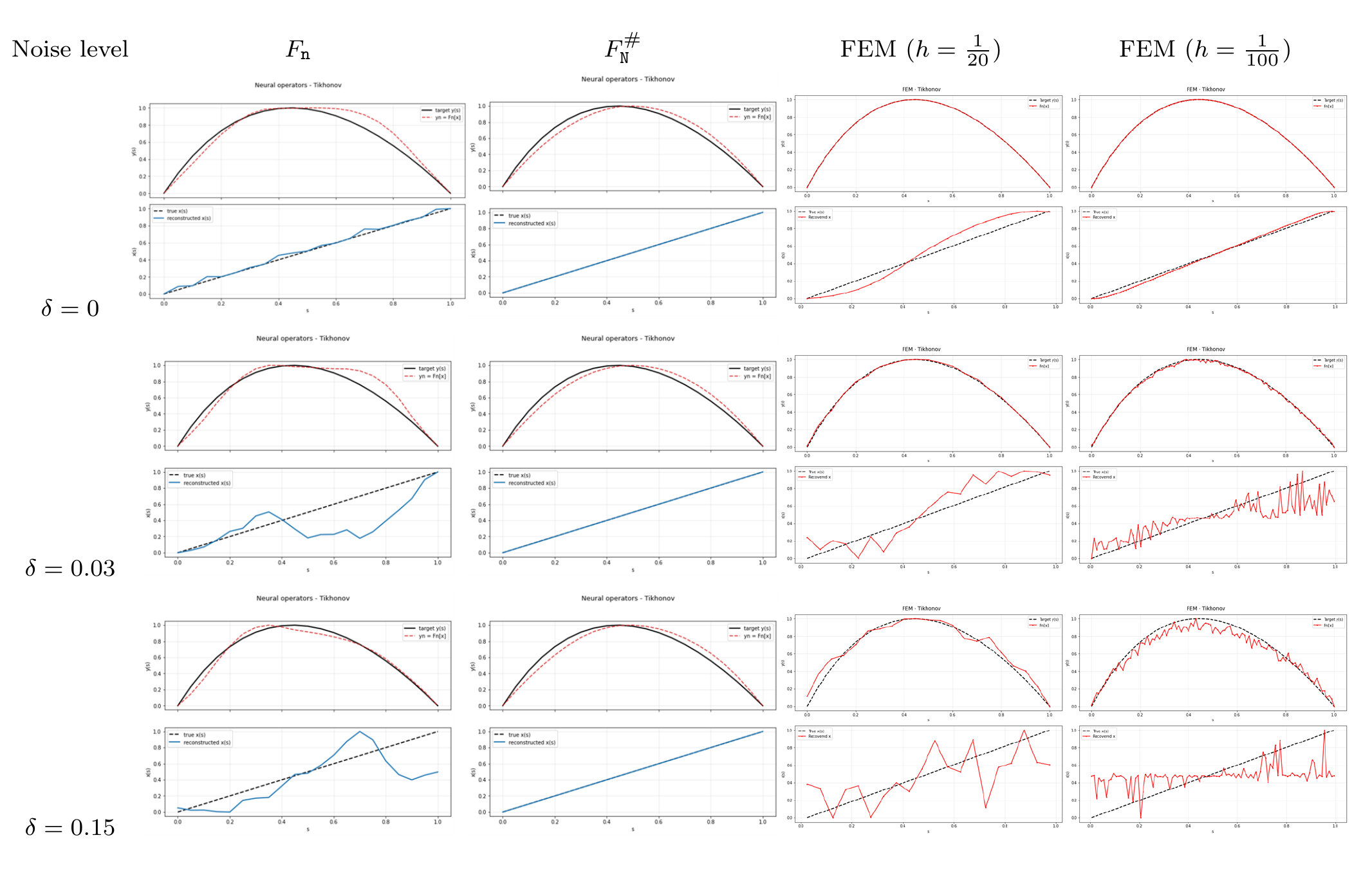}
    \caption{Results of solving the Tikhonov regularization problem with initial guess $\x^{(0)}(s) = s$ using a neural surrogate operator $F_\n$ (first image), approximation operator $F_\no^\#$ (second image), and the finite element method with 20 elements (third image) and 100 elements (fourth image); without noise (first row), with noise $\delta = 0.03$ (second row) and $\delta = 0.15$ (third row). The mean squared errors (MSE) of these results are recorded in \autoref{tab:mse} noise-free data and \autoref{tab:mse_noise} for data with noise-level $\delta = 0.03, 0.15$.}
    \label{fig:result_x0=s}
\end{figure}

\item \textbf{For \autoref{ex:2point}:} We solve the inverse problem using Tikhonov regularization, employing both a trained neural operator as a surrogate model and a finite element method. We consider two different initial guesses: $\x^{(0)} \equiv 0$ (\autoref{fig:c_result_x0=0}) and $\x^{(0)} \equiv 1$ (\autoref{fig:c_result_x0=1}). The results are then validated by comparing them with the analytic solution of 
\autoref{eq:invillpos} with ${\tt f} \equiv 1$,
\begin{equation} \label{eq:analytic_c}
    \begin{cases}
        \x(s) = 1\\
        \y(s) = \frac{-\exp(1-s)-\exp(s) +1 + \exp(1)}{1+\exp(1)},
    \end{cases}
\end{equation}
 which serves as the ground truth.

\begin{figure}[H]
    \centering
\includegraphics[width=1\linewidth]{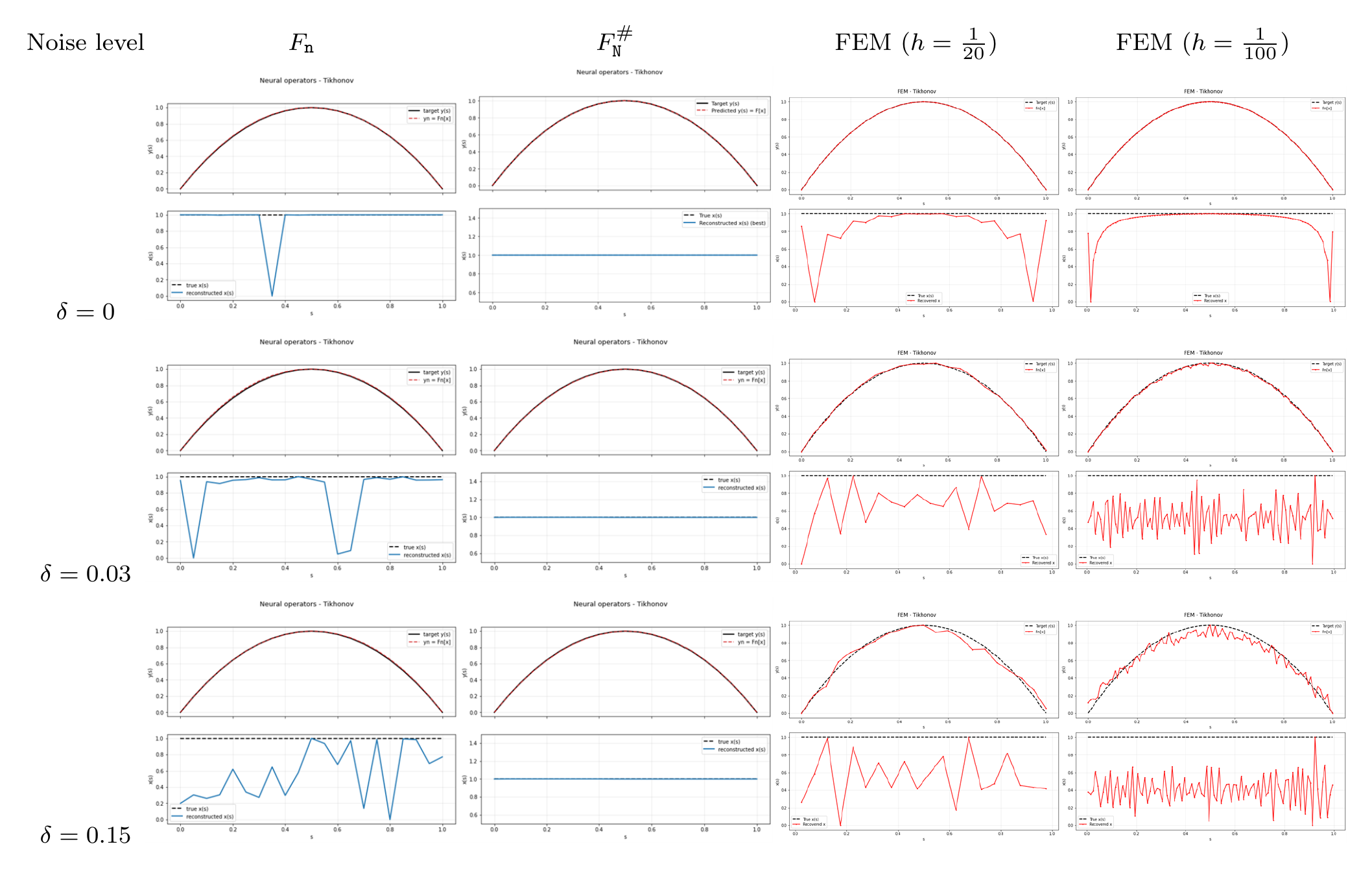}
    \caption{Results of solving the Tikhonov regularization problem with initial guess $\x^{(0)} \equiv 0$ using a neural surrogate operator $F_\n$ (first image), approximation operator $F_\no^\#$ (second image), and the finite element method with 20 elements (third image) and 100 elements (fourth image); without noise (first row), with noise $\delta = 0.03$ (second row) and $\delta = 0.15$ (third row).}
    \label{fig:c_result_x0=0}
\end{figure}

\begin{figure}[H]
     \centering
\includegraphics[width=1\linewidth]{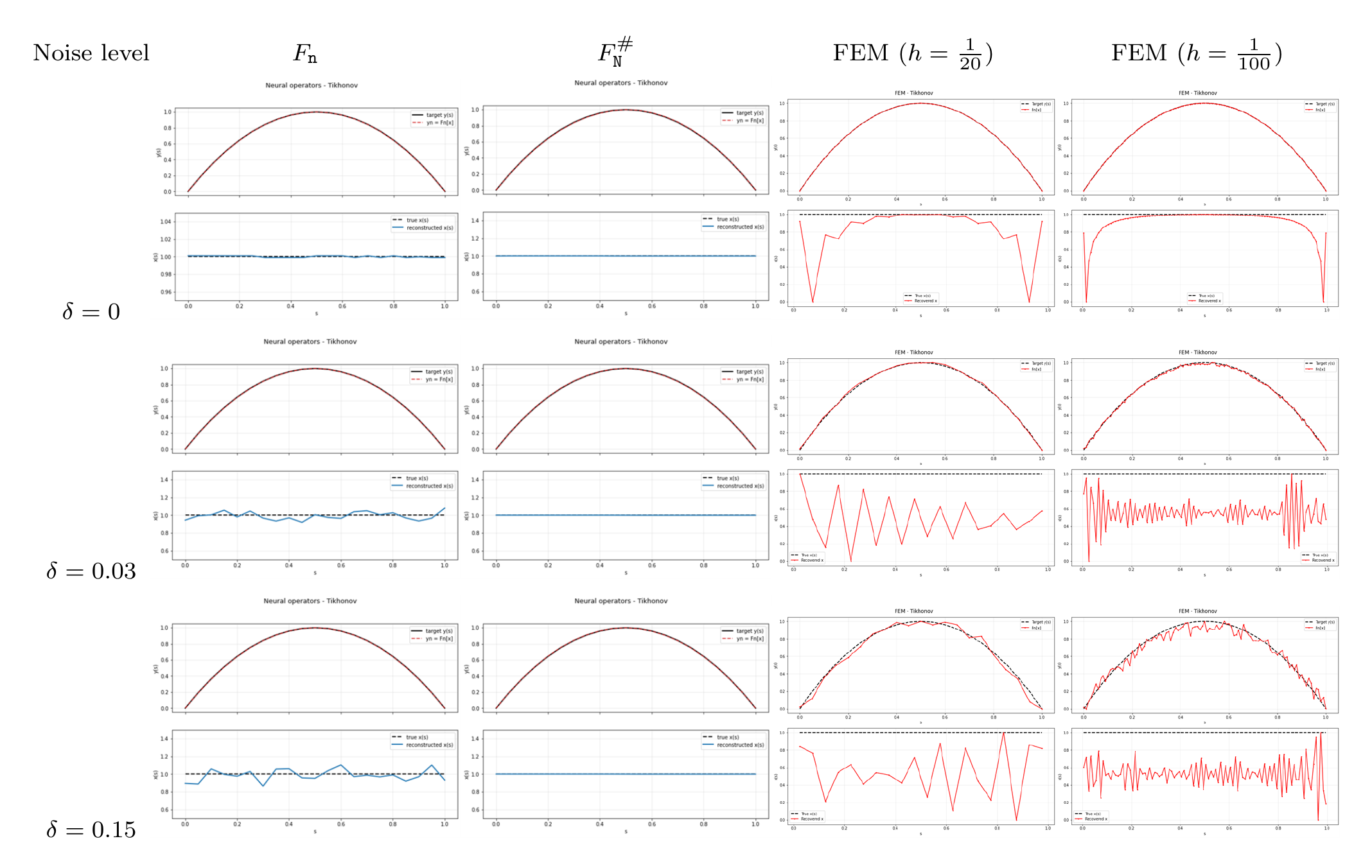}
    \caption{Results of solving the Tikhonov regularization problem with initial guess $\x^{(0)} \equiv 1$ using a neural surrogate operator $F_\n$ (first image), approximation operator $F_\no^\#$ (second image), and the finite element method with 20 elements (third image) and 100 elements (fourth image); without noise (first row), with noise $\delta = 0.03$ (second row) and $\delta=0.15$ (third row).}
    \label{fig:c_result_x0=1}
\end{figure}
\end{enumerate}

\begin{table}[H]
    \centering \footnotesize
    \begin{tabular}{c|c|c|c|c}
        & $||F_\no^\#[\x] - F[\x]||_{L^2}^2$ &  $||F_\n[\x] - F[\x]||_{L^2}^2$ & Data generation time &  Training time for $F_\n$ (s) \\ \hline
         \autoref{ex: darcyflow} & $2.6 \times 10^{-3}$ & $3.2 \times 10^{-3}$ & $0.006$ & 8 \\ 
         \autoref{ex:2point} & $1 \times 10^{-6}$ & $3.7 \times 10^{-6}$ & $0.006$ &  9\\
    \end{tabular}
    \caption{Mean squared error $||F_\no^\#[\x] - F[\x]||_{L^2}^2$ and $||F_\n[\x] - \y||_{L^2}^2$, measuring the discrepancy between the surrogate predictions $F_\no^\#[\x], F_\n[\x]$ and the reference data $\y=F[\x]$, respectively, along with the data generation and the training time of $F_\n$ in seconds (s). Here, $\no = 20$ and $\n = J(K(L + 2) + n + 1) + Lm \approx 10^{7}$, with $J = K = L = 200$ and $n = m = 1$.}
    \label{tab:mse_F}
\end{table}

From \autoref{tab:mse_F}, the surrogate operator $F_\n$ yields a small operator approximation error for both examples. The total training time for the operator $F_\n$, including data generation, is approximately 10 seconds. This indicates that, for the present problems, the model can be trained efficiently and that the resulting surrogate provides an accurate approximation of the forward operator at a low training cost.

The operator $F_\no^\#$ achieves slightly smaller error than $F_\n$ for both examples. Moreover, it is computationally more efficient, as it does not require any training.

However, both surrogate methods $F_\no^\#$ and $F_\n$, rely on the availability of training data. Although data generation is relatively fast for the two problems considered here, this may not be the case for other applications. In high-dimensional settings, synthetic data generation may be expensive if each forward evaluation is costly. Therefore, the surrogate framework is most suitable when data generation is inexpensive or when measurement or simulation data are already available, allowing the training cost to be effectively amortized over multiple uses.

\begin{table}[H]
    %Noise-free
    \hspace{0.55cm}\centering \footnotesize
    \begin{subtable}{0.48\textwidth}
        \centering
        \begin{tabular}{c|c|c|c}
            $\delta = 0$ & $\x^{(0)}(s) = 0$ & $\x^{(0)}(s) = 1$ & $\x^{(0)}(s) = s$\\ \hline
            $F_\n$ & 0.0045 & 0.0054 & 0.0006\\ \hline
            $F_\no^\#$ & 0.0255 & 0.0290 & 0.0000 \\ \hline
            FEM ($h = \frac{1}{20}$) & 0.0047 & 0.0537 & 0.0084 \\ \hline
            FEM ($h = \frac{1}{100}$) & 0.0793 & 0.1239  & 0.0007\\
        \end{tabular}
        \caption{\autoref{ex: darcyflow}, $\delta = 0$}
    \end{subtable}
    \hfill
    \begin{subtable}{0.48\textwidth}
        \centering
        \begin{tabular}{c|c|c}
            $\delta = 0$ & $\x^{(0)}(s) = 0$ & $\x^{(0)}(s) = 1$ \\ \hline
            $F_\n$  & 0.0476 & 0.0000\\ \hline
            $F_\no^\#$ & 0.0000 & 0.0000 \\ \hline
            FEM ($h = \frac{1}{20}$) & 0.1157 & 0.1156 \\ \hline
            FEM ($h = \frac{1}{100}$) & 0.0369 & 0.0287 \\
        \end{tabular}
        \caption{\autoref{ex:2point}, $\delta = 0$}
    \end{subtable}
    \caption{Mean squared error $||\x_\n^{\alpha\delta\nu} - \x||_{L^2}^2$ (for \autoref{ex: darcyflow}) and $||\x_{\n\xi}^{\alpha\delta\nu} - \x||_{L^2}^2$  (for \autoref{ex:2point}) between the reconstructed solutions $\x_\n^{\alpha\delta\nu}$, $\x_{\n\xi}^{\alpha\delta\nu}$ and the analytical solution $\x$ for noise-free input $\y$ ($\delta = 0$). Here, $\no = 20$ and $\n \approx 10^7$.}
    \label{tab:mse}
\end{table}

\autoref{tab:mse} compares the neural surrogates and FEM in terms of the Tikhonov reconstruction error.
\begin{itemize}   
    \item For \autoref{ex: darcyflow}, the surrogate $F_\n$ produces a lower mean squared error (MSE) than FEM for all tested initial guesses. In particular, the reconstruction error obtained with the neural operator $F_\n$ remains below $10^{-2}$ for all initializations considered, whereas FEM achieves errors below $10^{-2}$ only when the initial guess is sufficiently close to the exact solution. This indicates a stronger dependence of the FE-based optimization on the choice of initialization.
    
    Moreover, for FEM, refining the mesh from  $\frac{1}{20}$ to $\frac{1}{100}$ does not improve the reconstruction accuracy for the two initial guesses $\x^{(0)} \equiv 0$ and $\x^{(0)} \equiv 1$, but instead makes it worse. This suggests that the dominant source of error is related to the inverse problem itself, rather than to discretization. Finer discretization is not always beneficial for inverse problems. In fact, results in \cite{NeuSch90} show that the discretization of $\x$ should be a magnitude finer than for $\y$. Such an analysis, however, requires to balance one more additional parameter for the discretization of $\x$, which we left out in this paper.

    For the surrogate $F_\no^\#$, the MSE value is higher than to that of FEM with mesh size $\frac{1}{20}$ when using the initial guess $\x^{(0)} \equiv 0$. However, for the initial guesses $\x^{(0)} \equiv 1$ and $\x^{(0)}(s) = s$, $F_\no^\#$ provides better reconstructions than FEM with mesh size $\frac{1}{20}$, and also achieves lower MSE than FEM with mesh size $\frac{1}{100}$ for all initial guesses. Notably, with the initial guess $\x^{(0)}(s) = s$, $F_\no^\#$ yields an exact reconstruction. Similar to FEM, $F_\no^\#$ also exhibits a strong dependence on the choice of initialization. Nevertheless, $F_\no^\#$ generally performs worse than $F_\n$ in reconstruction, as it has fewer coefficients and is less flexible in representing the underlying solution.

    \item For \autoref{ex:2point}, the surrogate $F_\n$ produces a lower MSE than FEM with mesh $h = 1/20$, but a higher MSE than $h = 1/100$ for the initial guess $\x^{(0)} \equiv 0$. Surprisingly, as contrast to \autoref{ex: darcyflow}, in this example, the surrogate $F_\no^\#$ performs exceptionally well, providing (nearly) exact reconstructions for all initial guesses. However, the result depend on the choice of learning rate, which requires additional time for parameter tuning. Notably, for $\x^{(0)} \equiv 1$, even when the initial guess coincides with the exact solution, FEM does not yield an accurate reconstruction, whereas the surrogates $F_\n$ and $F_\no^\#$ provide (nearly) exact reconstructions. 

    Moreover, for this problems, refining the mesh improves the reconstruction accuracy for FEM. We also observe that without mollification, the reconstructions using surrogates $F_\no^\#$ and $F_\n$ exhibit strong oscillations, highlighting the importance of the mollification process for $L^2$-solutions.

	\item  The reconstruction using the neural operator $F_\n$ as a surrogate, despite following the overall shape of the true solution, is slightly oscillatory and sometimes shows a localized peak. This happens because a stochastic optimizer (AdamW) is used to train the operator and to minimize the Tikhonov regularization functional. In AdamW, gradients are estimated using mini-batches, which makes the computation faster but introduces additional error in the gradient estimates because each batch is only a small random subset of the data. This randomness can help generalization, but it also leads to occasional updates in the wrong direction. The effect becomes stronger when the data are noisy, since the gradients become more unstable, leading to larger fluctuations in the reconstruction. 

	  When using $F_\no^\#$ as a surrogate, a similar oscillatory behavior can also be observed, since the Tikhonov regularization functional is also minimized stochastically. However, because the model is linear and has fewer degrees of freedom, it has less flexibility to adapt to fluctuations in the stochastic updates. As a result, the oscillations are less pronounced and less frequent, and the reconstruction appears more stable and visually smoother. Moreover, it is less sensitive to noise in the input data because the linear structure restricts how the solution can respond to perturbations.
\end{itemize}

\begin{table}[H]
    %Noise 0.03
    \hspace{0.55cm}\centering \footnotesize
    \begin{subtable}{0.48\textwidth}
        \centering
        \begin{tabular}{c|c|c|c}
            $\delta = 0.03$ & $\x^{(0)}(s) = 0$ & $\x^{(0)}(s) = 1$ & $\x^{(0)}(s) = s$\\ \hline
            $F_\n$ & 0.0118 & 0.0304 & 0.0666\\ \hline
            $F_\no^\#$ & 0.0266 & 0.0316 & 0.0000 \\ \hline
            FEM ($h = \frac{1}{20}$) & 0.0675 & 0.0904 & 0.0196 \\ \hline
            FEM ($h = \frac{1}{100}$) & 0.1368 & 0.1390  & 0.0271\\
        \end{tabular}
        \caption{\autoref{ex: darcyflow}, $\delta = 0.03$}
    \end{subtable}
    \hfill
    \begin{subtable}{0.48\textwidth}
        \centering
        \begin{tabular}{c|c|c}
            $\delta = 0.03$ & $\x^{(0)}(s) = 0$ & $\x^{(0)}(s) = 1$ \\ \hline
            $F_\n$ & 0.1315 & 0.1011\\ \hline
            $F_\no^\#$ & 0.0000 & 0.0000 \\ \hline
            FEM ($h = \frac{1}{20}$) & 0.1850 & 0.2429 \\ \hline
            FEM ($h = \frac{1}{100}$) & 0.2686 & 0.2276 \\
        \end{tabular}
        \caption{\autoref{ex:2point}, $\delta = 0.03$}
    \end{subtable}

    %noise 0.15
     \hspace{0.55cm}\centering \footnotesize
    \begin{subtable}{0.48\textwidth}
        \centering
        \begin{tabular}{c|c|c|c}
            $\delta = 0.15$ & $\x^{(0)}(s) = 0$ & $\x^{(0)}(s) = 1$ & $\x^{(0)}(s) = s$\\ \hline
            $F_\n$ & 0.0166 & 0.0457 & 0.0585\\ \hline
            $F_\no^\#$ & 0.0283 & 0.0019 & 0.0000 \\ \hline
            FEM ($h = \frac{1}{20}$) & 0.1173 & 0.1302 & 0.0656 \\ \hline
            FEM ($h = \frac{1}{100}$) & 0.1084 & 0.1017  & 0.0679\\
        \end{tabular}
        \caption{\autoref{ex: darcyflow}, $\delta = 0.15$}
    \end{subtable}
    \hfill
    \begin{subtable}{0.48\textwidth}
        \centering
        \begin{tabular}{c|c|c}
            $\delta = 0.15$ & $\x^{(0)}(s) = 0$ & $\x^{(0)}(s) = 1$ \\ \hline
            $F_\n$ & 0.2887 & 0.0045\\ \hline
            $F_\no^\#$ & 0.0000 & 0.0000 \\ \hline
             FEM ($h = \frac{1}{20}$) & 0.3389 & 0.2478 \\ \hline
            FEM ($h = \frac{1}{100}$) & 0.3813 & 0.2486 \\
        \end{tabular}
        \caption{\autoref{ex:2point}, $\delta = 0.15$}
    \end{subtable}
    \caption{Mean squared error $||\x_\n^{\alpha\delta\nu} - \x||_{L^2}^2$ (for \autoref{ex: darcyflow}) and $||\x_{\n\xi}^{\alpha\delta\nu} - \x||_{L^2}^2$  (for \autoref{ex:2point}) between the reconstructed solutions $\x_\n^{\alpha\delta\nu}$, $\x_{\n\xi}^{\alpha\delta\nu}$ and the analytical solution $\x$ for perturbed inputs $\y^\delta$ with noise levels $\delta = 0.03, 0.15$. Here, $\no = 20$ and $\n \approx 10^7$.}
    \label{tab:mse_noise}
\end{table}

From \autoref{tab:mse_noise}, we observe that $F_\no^\#$ is the most stable method with respect to varying noise levels, as its MSE values remain approximately the same across differen noise level $\delta = 0, 0.03, 0.15$ for both examples. In contrast, the performance of $F_\n$ deteriorates as the noise level increases for \autoref{ex:2point}. For \autoref{ex: darcyflow}, the error increases compared to the noise-free case, but remains similar across the considered noise levels. The FEM approach, however, degrades drastically under noise. 

Furthermore, we observe that using a mesh size of $h = 1/100$ leads to an overly fine discretization, which induces oscillatory behavior in the reconstruction for noisy input. Even if the data is noise-free, increasing the number of mesh points to $100$ most of the time does not improve the reconstruction.

\begin{table}[H]
    \begin{subtable}{0.48\textwidth}
    \centering \footnotesize
    \begin{tabular}{c|c|c}
        &  $||F_\n[\x] - F[\x]||_{L^2}^2$ & $||F_\no^\#[\x] - F[\x]||_{L^2}^2$ \\ \hline
         $\no = 10$ & $3.8 \times 10^{-3}$ & $3.1 \times 10^{-3}$\\
         $\no = 20$ & $3.2 \times 10^{-3}$ & $2.6 \times 10^{-3}$ \\ 
         $\no = 30$ & $3.0 \times 10^{-3}$ & $2.3 \times 10^{-3}$\\ 
         $\no = 40$ & $3.7 \times 10^{-3}$ & $3.2 \times 10^{-3}$
    \end{tabular}
     \caption{\autoref{ex: darcyflow}}
    \end{subtable} 
    \hfill
    \begin{subtable}{0.48\textwidth}
    \centering \footnotesize
    \begin{tabular}{c|c|c}
        &  $||F_\n[\x] - F[\x]||_{L^2}^2$ & $||F_\no^\#[\x] - F[\x]||_{L^2}^2$ \\ \hline
         $\no = 10$ &  $3.9 \times 10^{-3}$ & $1 \times 10^{-6}$  \\
         $\no = 20$ &  $3.7 \times 10^{-6}$ & $1 \times 10^{-6}$   \\ 
         $\no = 30$ & $2.4 \times 10^{-6}$ & $1 \times 10^{-6}$    \\ 
         $\no = 40$ & $4.2 \times 10^{-6}$ & $1 \times 10^{-6}$ 
    \end{tabular}
     \caption{\autoref{ex:2point}}
    \end{subtable} 
    \caption{Mean squared error $||F_\no^\#[\x] - F[\x]||_{L^2}^2$ and $||F_\n[\x] - \y||_{L^2}^2$ computed for different numbers of training samples $\no = 10,20,30,40$, and $\n = J(K(L + 2) + n + 1) + Lm \approx 10^{7}$, with $J = K = L = 200$ and $n = m = 1$.} 
    \label{tab:mse_F_N}
\end{table}

A similar phenomenon is observed when increasing the number of training samples $\no$ for both surrogate operator $F_\n$ and $F_\no^\#$ (see \autoref{tab:mse_F_N}). Specifically, increasing $\no$ reduces the prediction error $||F_\n[\x] - \y||_{L^2}^2$ as $\no$ increases from 10 to 20 and 30, however, from $\no = 40$, the error saturates, and further increasing the training set does not yield any improvement. This behavior indicates diminishing returns with respect to the training data size, suggesting that the model has already captured the essential features of the underlying operator. Nevertheless, increasing the amount of training data does not improve the reconstruction error, and in some cases even leads to deterioration. This raises the question of how much training data is sufficient in practice and could be considered in the future.

Overall, the neural surrogate $F_\n$ provides more accurate and robust reconstructions than FEM across both examples, with lower sensitivity to the initial guess. $F_\no^\#$ offers a competitive alternative without any training cost and lower per-solve time, and it can provide better reconstructions than $F_\n$ when the initial guess is sufficiently close to the exact solution. However, it attains lower accuracy in other cases, reflecting a stronger dependence on the initial guess. This behavior is due to its smaller number of coefficients, which limits flexibility and prevents it from matching the reconstruction accuracy and stability of $F_\n$. These observations highlight a trade-off between computational cost and reconstruction accuracy.

Notably, although $F_\no^\#$ can achieve better accuracy in approximating the forward operator than $F_\n$, this does not always translate to the inverse problem. This highlights the instability of inverse problems, where higher forward accuracy does not necessarily guarantee improved reconstruction performance.

\section*{Conclusion}
\label{se:conclusion}

In this paper, we studied Tikhonov regularization for solving inverse problem, where neural operators are used as surrogates for the forward operator. The developed theory is based on the theory of finite dimensional approximations of Tikhonov regularization. These results show how to choose the regularization parameter in dependence of the discretization and the approximation error of the surrogate operator. Traditionally, for operator approximation, finite element and finite difference have been used, while here we study neural operators, which are computed from supervised training data.
We have explored the application of neural operators in regularization with two basic examples of inverse problems. For applications, we realized that the existing approximation results of neural operators are not sufficient for our purposes, because the standard regularization theory requires formulations in Sobolev and Lebesgue spaces, rather than spaces of continuous functions, and approximation rates. Our motivation of neural operators provides some insight in the structure of these operators and confirms results from the literature that the amount of representing parameters $\mathcal{T}$ can be significantly reduced in practice, which simplifies the highly complex training process.
The analysis of this paper is supported by a series of examples, where analytical computations give insight in the structure of neural operators. 

\subsection*{Acknowledgments}
This research was funded by the Austrian Science Fund
(FWF) 10.55776/P34981 -- New Inverse Problems of Super-Resolved Microscopy (NIPSUM) and 
SFB 10.55776/F68 ``Tomography Across the Scales'', project F6807-N36
(Tomography with Uncertainties). For open access purposes, the author has
applied a CC BY public copyright license to any author-accepted manuscript
version arising from this submission.
The financial support by the Austrian Federal Ministry for Digital and Economic
Affairs, the National Foundation for Research, Technology and Development and the Christian Doppler
Research Association is gratefully acknowledged. 

The authors wish to thank two referees for their valuable comments, which led to a significant 
improvement and enlargement of the manuscript.
%'
%' references
%'
\section*{References}
\renewcommand{\i}{\ii}
\printbibliography[heading=none]
%;

\end{document}